\documentclass[12pt,reqno]{amsart}

\usepackage[utf8]{inputenc}
\usepackage[english]{babel}
\usepackage[T1]{fontenc}
\usepackage{graphicx}
\usepackage{amssymb,amstext,amsthm}
\usepackage{enumerate}
\usepackage{geometry}
\usepackage{amsmath,amscd}
\usepackage{hyperref}
\usepackage[all]{xy}
\usepackage[normalem]{ulem}
\usepackage{calc,color}
\usepackage{epsfig,graphics}
\usepackage{pst-grad}
\usepackage{pst-plot}
\usepackage{footnote}

\newcommand{\R}{\ensuremath{\mathbb{R}}}

\newcommand{\Z}{\ensuremath{\mathbb{Z}}}
\newcommand{\N}{\ensuremath{\mathbb{N}}}

\newcommand{\e}{\epsilon}
\newcommand{\vp}{\varphi}
\newcommand{\la}{\lambda}

\def \tn {\textnormal}

\newcommand{\df}[1]{{\color{black} \bf #1}}
\newcommand{\norma}[1]{\|\!|#1\|\!|}

\newtheorem{theorem}{Theorem}[section]
\newtheorem{lemma}[theorem]{Lemma}
\newtheorem{corollary}[theorem]{Corollary}
\newtheorem{proposition}[theorem]{Proposition}
\newtheorem{definition}[theorem]{Definition}
\newtheorem{remark}[theorem]{Remark}
\numberwithin{equation}{section}

\title[Lipschitiz perturbations of MS semigroups]{Lipschitz perturbations of Morse-Smale semigroups}

\author[M.C. Bortolan]{M.C. Bortolan}
\author[C.A.E.N. Cardoso]{C.A.E.N Cardoso}
\author[A.N. Carvlho]{A.N. Carvalho}
\author[L. Pires]{L. Pires}

\begin{document}

\begin{abstract}
In this paper we will deal with Lipschitz continuous perturbations of Morse-Smale semigroups with only equilibrium points as critical elements. We study the behavior of the structure of equilibrium points and their connections when subjected to non-differentiable perturbations. To this end we define more general notions of \emph{hyperbolicity} and \emph{transversality}, which do not require differentiability.
\end{abstract}

\maketitle

\tableofcontents

\section{Introduction}

The study of the asymptotic behavior of autonomous dynamical systems is a rich research area that is being developed since more than seven decades with deep contributions from many different authors with several papers dedicated to this theory, such as \cite{BV83,CC2011,ChepVishik94,hale,Haraux91,Lady,MST,Rob,SY,Temam,Vi92}. Within this area there is the study of {\it continuity of the structure under perturbations}; that is, the field responsible to answer the following questions: can we transport properties from a dynamical systems to another which is close - in some sense - to the first? Also, if we have properties of a family of dynamical systems which are close to a given dynamical system, can we transport the properties of the family to the ``limiting system'' ?

These questions have real value when dealing with systems describing real world phenomena, since due to approximations and the use of empirical laws, such systems are always approximations of the real problem. Hence, to be able to study the mathematical model and conclusively give informations about the real system, one must be certain that we can transport the obtained properties to the real system; that is, we need to be sure that we have some kind of ``continuity'' among the dynamical systems, if we want to give informations about the asymptotic behavior of the real system. Here, more recently, we have had several papers that treat this issue, such as \cite{aragao,aragao2,CbLR,CL1,CL09,CLR09,JG,Oliva,Raugel,Ryba}.

Before continuing, we just present simple definitions that will allow us present some results on this field. Let $(X,d_X)$ be a metric space, $C(X)$ the set of continuos maps from $X$ into itself, $\mathbb{T}=\mathbb{Z}$ or $\mathbb{R}$ and $\mathbb{T}^+=\{t\in \mathbb{T}\colon t\geqslant 0\}$.

\begin{definition} \label{def:Semigroup}
A one-parameter family $\{T(t)\colon t\in \mathbb{T}^+\}\subset C(X)$ is called an \df{autonomous dynamical system}, or simply a \df{semigroup}, in $X$ if:
\begin{itemize}
\item[\bf (i)] $T(0)x=x$ for each $x\in X$;
\item[\bf (ii)] $T(t+s)=T(t)T(s)$ for all $t,s\in \mathbb{T}^+$;
\item[\bf (iii)] the map $\mathbb{T}^+\times X\ni (t,x)\mapsto T(t)x\in X$ is continuous.
\end{itemize}

When $\mathbb{T}=\mathbb{Z}$ we say that $\{T(t)\colon t\in \mathbb{T}^+\}$ is a \df{discrete semigroup}.
\end{definition}

\begin{remark} \label{rem:DisSemigroup}
Clearly for $ \ \mathbb{T}=\mathbb{Z}$ the condition {\rm (iii)} in Definition \ref{def:Semigroup} is automatically satisfied. Also in the case of discrete semigroups, defining $T=T(1)$, we have $T(n)=(T(1))^n=T^n$ for each $n\in \mathbb{Z}^+$ by condition {\rm (ii)} and the discrete semigroup is the family of maps $\{T^n\colon n\in \mathbb{Z}^+\}$.
\end{remark}

We say that $A \subset X$ is \df{invariant} for $\{T(t)\colon t\in \mathbb{T}\}$ if $T(t)A=A$ for each $t\in \mathbb{T}^+$. Also, for $A,B\subset X$ we say that $A$ \df{attracts} $B$ under the action of $\{T(t)\colon t\in \mathbb{T}^+\}$ if
$$
{\rm dist}_H(T(t)B,A)\to 0, \hbox{ as } t\to  \infty \hbox{ in } \mathbb{T}^+,
$$
where ${\rm dist}_H(\cdot,\cdot)$ denotes the Hausdorff semidistance between two sets; that is, if $C,D\in X$ are non-empty subsets we have
$$
{\rm dist}_H(C,D) = \sup_{c\in C}\sup_{d\in D}d_X(c,d).
$$

\begin{definition}
A compact subset $\mathcal{A}$ of $X$ is a \df{global attractor} of $\{T(t)\colon t\in \mathbb{T}^+\}$ if it is invariant for $\{T(t)\colon t\in \mathbb{T}\}$ and attracts all bounded subsets of $X$ under the action of $\{T(t)\colon t\in \mathbb{T}\}$.
\end{definition}

Hence with this definition we can make more clear what we mean by {\it asymptotic behavior} for an autonomous dynamical system. It is easy to see that each semigroup $\{T(t)\colon t\in \mathbb{T}^+\}$ has at most one global attractor $\mathcal{A}$, and this global attractor attracts all the {\it orbits} $\{T(t)x\colon t\in \mathbb{T}^+\}$ for $x\in X$, hence the global attractor is the `limiting object' of all possible trajectories of our semigroup and thus the behavior of $\{T(t)\colon t\in \mathbb{T}^+\}$ for $t\to \infty$ in $\mathbb{T}^+$  is described precisely by the global attractor $\mathcal{A}$.

We will see that in fact the global attractor is more than that; that is, the global attractor also {\it contains} all possible bounded trajectories that can be defined for all $t\in \mathbb{T}^+$. To see this property we define a \df{global solution} of $\{T(t)\colon t\in \mathbb{T}^+\}$ as a continuous function $\xi\colon \mathbb{T}\to X$ such that
$$
T(t)\xi(s) = \xi(t+s), \hbox{ for all } s\in \mathbb{T} \hbox{ and } t\in \mathbb{T}^+.
$$

If $\xi(0)=x$, we say that $\xi$ is a \df{global solution through} $x$. And with these definitions, when the semigroup $\{T(t)\colon t\in \mathbb{T}\}$ has a global attractor $\mathcal{A}$, we have
$$
\mathcal{A} = \{x\in X\colon \hbox{ there exists a bounded global solution through }x\}.
$$

Hence the global attractor $\mathcal{A}$ of a semigroup is {\it the} object to study if one wants to understand the behavior of $\{T(t)\colon t\in \mathbb{T}^+\}$ as $t\to \infty$ in $\mathbb{T}^+$.

For the issue of ``continuity'' described before, if we have a family of semigroups $\{T_\eta(t)\colon t\in \mathbb{T}^+\}$ with a global attractor $\mathcal{A}_\eta$ for $\eta\in [0,1]$, there is a question quite simple to present: for a suitable convergence of $T_\eta$ to $T_0$, what kind of convergence can we expect for $\mathcal{A}_\eta$ to $\mathcal{A}_0$ as $\eta\to 0^+$?

This question has been answered in several papers throughout the years, and following their main results, we can outline a rough sketch of what kind of continuity we can obtain to the global attractors $\mathcal{A}_\eta$ as upper semicontinuity, lower semicontinuity, topological stability and geometrical stability.

In this paper, our focus will be on the {\it geometrical stability}, which is the stability of the energy levels of the invariants inside the attractor and it is obtained with great effort for Morse-Smale systems, see \cite{hmo,Palis,Oliva} for instance. More precisely, we will deal with the following problem: to achieve geometrical stability, the hypothesis of differentiability for the family of semigroups $\{T_\eta(t)\colon t\in \mathbb{T}^+\}$ always appear in the literature presented so far in this topic, but can it be obtained {\it without it}?

This question is vey important to consider, since when we are dealing with evolution equations in some $L^p$-space, there are no differentiable functions from $L^p$ into itself, other then linear functions. Hence, there is no hope to achieve differentiability of the generated semigroup. Hence, the theory of geometrical stability so far, only allows us to consider differentiable semigroups with differentiable perturbations, which is far from being a reality in $L^p$. Our main goal in this paper is to introduce definitions and results about geometrical stability for {\it Lipschitz continuous perturbations}.  We will deal with the case of {\it discrete semigroups}, and by Remark \ref{rem:DisSemigroup} it is enough to consider a map $T\in C(X)$.

In Section \ref{sec:InvManifolds} we describe preliminary results concerning Lipschitz global and local invariant manifolds near a fixed point, more specifically Theorems \ref{teoUS} and \ref{teo:LIM}, and here we stress that these results can be found in the manuscripts\footnote{Handwritten notes: ``Invariant manifolds near a fixed point'', of D. Henry} of D. Henry, and we do not claim their authorship. However, since the notation introduced in the proof is important for our work, for the sake of completeness we decided to include the proofs in our work. 

Section \ref{sec:LHP} is where we introduce the main concept of our study, the {\it $\mathfrak{L}$-hyperbolic points} (see Definition \ref{hyperbolic_fixed_point}), and to this end, we must introduce first the concept of {\it weakly hyperbolic point} (see Definition \ref{w_hyperbolic_fixed_point}). Also in the section, we state a result concerning invariant manifolds near a weakly hyperbolic point (Theorem \ref{cor:US}), which is a direct corollary of Theorem \ref{teo:LIM}, as well as study the {\it isolation} of $\mathfrak{L}$-hyperbolic points, as seen in Theorem \ref{hu1}. With this property of isolation, we can describe the behavior of solutions near a $\mathfrak{L}$-hyperbolic point, which is described in Theorem \ref{hu1*}.

To continue our study, we devote Section \ref{sec:APHP} to the autonomous perturbations of $\mathfrak{L}$-hyperbolic points, that is, we analyze the permanence of $\mathfrak{L}$-hyperbolic points when submitted to small (in the sense of Definition \ref{def:norms}) autonomous perturbations (see Theorem \ref{teo:CHP}). Section \ref{sec:APIM} has the same direction as Section \ref{sec:APHP}, but we study the permanence of invariant manifolds under small autonomous perturbations and we obtain Theorem \ref{convvarfinal*}.

In Section \ref{sec:TS} we describe the {\it topological stability} of dynamically gradient maps, which is a recent topic of research (see \cite{aragao,aragao2,CL09} for instance), and deals with the permanence of inner structures of the attractors for maps, under small autonomous perturbations.  Mainly, we use the result in these papers adapted to our case, and obtain Proposition \ref{prop:TS}.

Section \ref{sec:Trans} is devoted to the study of a new concept, the {\it $\mathfrak{L}$-transversality}, which is a replacement for the usual transversality property, and is key to obtain permanence of intersections. The main result of this section is Proposition \ref{lem6.3}.

In Section \ref{sec:LMSS} we are able to finally present the main concept of our work, see Definition \ref{WLSM},  the concept of {\it $\mathfrak{L}$-Morse-Smale maps}, and also the main result of our paper, Theorem \ref{st3}, which gives the permanence of $\mathfrak{L}$-Morse-Smale maps under small autonomous perturbations. Section \ref{sec:Ex} is devoted to present one example to illustrate all the theory developed throughout the paper. 

Finally, our work has also an appendix, divided in three sections. Section \ref{app:A} is devoted to the proofs of Proposition \ref{hu3} and Lemma \ref{convvarlipfinal_lem}, which are fairly technical and would just fog the view of the general outline of the paper. In Section \ref{app:LipschitzMani} we present the proof of Proposition \ref{prop6.9}, as well as the necessary definitions and previous results required. Lastly, Section \ref{app:Diff} is dedicated to the study of differentiability of Nemytskii operators, which is necessary to validate the importance of the theory developed in the work, once we realize that there are almost none differentiable function from $L^p(\Omega)$ into $L^p(\Omega)$, when we consider functions arising from forcing terms of differential equations, and we use the results of this section in the example of Section \ref{sec:Ex}.

\section{Invariant manifolds near a fixed point} \label{sec:InvManifolds}

Let $(X,\|\cdot\|)$ be a Banach space, $\mathcal{L}(X)$ be the set of bounded linear operators of $X$ into itself and $L\in \mathcal{L}(X)$ be a bounded linear operator such that $\sigma(L) \cap \{\xi \in \mathbb{C}\colon |\xi|=\rho\}=\varnothing$, for some $\rho>0$. If $\Gamma(t)= \rho e^{it}$, $t\in [0,2\pi]$, we have the spectral projections
$$
\pi_s = \frac{1}{2\pi i} \int_{\Gamma} (\lambda-L)^{-1}d\lambda\quad \hbox{ and } \quad \pi_u=I-\pi_s,
$$
that give us a decomposition of $X$ and $L$ as $X=X_u\oplus X_s$, $L=L_u\oplus L_s$, respectively, where $X_j=\pi_j X$ and $L_j=L|_{X_j}\colon X_j\rightarrow X_j$, for $j=u,s$. Then
\begin{equation} \label{eq:AB}
\sigma(L_s)\subset \{\lambda\in \mathbb{C}\colon |\lambda|<b\} \hbox{ and }\sigma(L_u)\subset \{\lambda\in \mathbb{C}\colon |\lambda|>a\}, \hbox{ for some } 0<b<\rho<a,
\end{equation}
and we can choose norms in $X_u$ and $X_{s}$ (which we denote the same $\|\cdot\|$) such that $\|L_s\|\leqslant b$ and $\|L_u^{-1}\|\leqslant a^{-1}$, with these norms equivalent to the norm induced by the norm of $X$. Therefore, in $X$ we use the equivalent norm (again, we denote it by $\|\cdot\|$) given by
$$
\|x_u+x_{s}\|=\max \{\|x_u\|,\|x_{s}\|\}.
$$

Also, in the Banach space $X$, for $A\subset X$ non-empty we define
$$
d(x,A)=\inf_{a\in A}\|x-a\|,
$$
the usual distance between points and sets and if $A\subset X$ and $r>0$, we denote the $r$-neighborhood of $A$ in $X$ by
$$
\mathcal{O}_r(A)=\{x\in X\colon d(x,A)<r\}.
$$
and for each $y\in X$ we denote $B_r^X(y)=\mathcal{O}_r(\{y\})$.


\begin{definition}
\label{def:SmallLConstant}
Let ${L}\in \mathcal{L}(X)$ and $\rho,b,a>0$ be described as above and $U\subset X$ an open set. We say a map $N\colon U\to X$ has
\df{small Lipschitz constant with respect to $L$}
if there exists $\gamma>0$ such that $b+2\gamma<\rho<a-2\gamma$ and
$$
\|N(x)-N(y)\|\leqslant \gamma \|x-y\|, \hbox{ for all }x,y \in U.
$$

For any Lipschitz map $N$, ${\rm Lip}(N)$ denotes \emph{any} Lipschitz constant of $N$.
\end{definition}

\begin{definition} \label{def:EqPoint}
We say that a function $\xi\colon \mathbb{Z}\to X$ is a \df{global solution} for a map $T\in C(X)$ if $T(\xi(m))=\xi(m+1)$ for each $m\in \mathbb{Z}$.
We say that a point $x\in X$ is an \df{equilibrium point} for a map $T\in C(X)$ if $T(x)=x$.
\end{definition}

If $\xi$ is a global solution for $T$, then it is a global solution for the {\it discrete semigroup} $\{T^n\colon n\in \mathbb{N}\}$, since
$$
T^n(\xi(m))= \xi(m+n), \hbox{ for all }m\in \mathbb{Z} \hbox{ and } n\in \mathbb{N}.
$$

Clearly, if $x$ is an equilibrium point for $T$, the function $\xi(m)=x$ for all $m\in \mathbb{Z}$ is a global solution for $T$. In general, in this case, we use that words {\it equilibrium point} and {\it stationary solution} indistinctly.


In the description presented in the introduction, we can see that to study continuity of attractors and to go further than upper semicontinuity, we must be able to describe and obtain properties on the invariant structures inside the global attractors. We present here a result that has this exact purpose; that is, a result that characterizes the invariant manifolds of a map $S=L+N$ with $L,N$ satisfying the Definition \ref{def:SmallLConstant}. This result can be found in the manuscripts of D. Henry and for the sake of completeness, both from the result itself and notations therein, we present the proof.

%
%


\begin{theorem}[Global invariant manifolds]
\label{teoUS}
Let $X$ be a Banach space, $L:X\rightarrow X$ a bounded linear operator, $\rho>0$ such that $\sigma(L) \cap \{\lambda \in \mathbb{C}\colon |\lambda|=\rho\}=\varnothing$. Then there exist $\gamma >0$, a decomposition $X=X_u\oplus X_{s}$ and an equivalent norm in $X$ such that if $N\colon X\rightarrow X$ is Lipschitz continuous and satisfies $N(0)=0$ and ${\rm Lip}(N)\leqslant \gamma$, there exist sets $W^{\rm s}_{\rho}$, $W^{\rm u}_{\rho}$ in $X$ such that for $S=L+N$, we have the following:
\begin{enumerate}
\item[{\bf (a)}] $W^{\rm u}_{\rho}$ is a graph of a Lipschitz map over $X_u$, $W^{\rm s}_{\rho}$ is a graph of a Lipschitz map over $X_{s}$.
\item[{\bf (b)}] $SW^{\rm u}_{\rho}=W^{\rm u}_{\rho}$, $SW^{\rm s}_{\rho}=W^{\rm s}_{\rho}\cap S(X)$ and the restriction $S|_{W^{\rm u}_{\rho}}$ is a homeomorphism.
\item[{\bf (c)}] $W^{\rm s}_{\rho}\cap W^{\rm u}_{\rho}=\{0\}$.
\item[{\bf (d)}] ${\rm Lip}\left(S|_{W^{\rm s}_{\rho}}\right)<\rho$ and ${\rm Lip}\left(S|_{W^{\rm u}_{\rho}}^{-1}\right)<\rho^{-1}$.
\item[{\bf (e)}] We have
\begin{equation*}
\begin{split}
&\quad {\bf 1.} \ W^{\rm s}_{\rho}=\{x\in X\colon \rho^{-n} \|S^nx\|\to 0 \hbox{ as }n\to \infty\},\\
&\quad {\bf 2.} \ W^{\rm u}_\rho=\{x\in X\colon \hbox{ there exists } \{x_j\}_{j\leqslant 0} \hbox{ with }x_0=x, x_{j+1}=S(x_j) \hbox{ and } \sup_{j\leqslant 0}\rho^{-j}\|x_j\|<\infty\},
\end{split}
\end{equation*}
and for $\rho^{\ast}$ sufficiently close to $\rho$, $W^{\rm s}_{\rho^{\ast}}=W^{\rm s}_{\rho}$ and $W^{\rm u}_{\rho^{\ast}}=W^{\rm u}_{\rho}$;
\end{enumerate}

Moreover, if $L$ is an isomorphism, then for $\gamma$ sufficiently small, $S$ is a homeomorphism and $SW^{\rm s}_{\rho}=W^{\rm s}_{\rho}$.
\end{theorem}
\begin{proof}
We have a decomposition $X=X_u\oplus X_{s}$, $L=L_u\oplus L_s$ and $a,b,\gamma>0$ satisfying \eqref{eq:AB}, and we can define the norm $\|x_u+x_{s}\|=\max\{\|x_u\|,\|x_{s}\|\}$, which is equivalent to the initial norm in $X$. In this decomposition, $S$ takes the form
\begin{equation*}
\begin{split}
S\colon &X_u\times X_{s}\rightarrow X_u\times X_{s}\\
     &(x_u,x_{s})\mapsto(L_ux_u+N_u(x_u+x_{s}),L_sx_{s}+N_s(x_u+x_{s})),
\end{split}
\end{equation*}
with ${\rm Lip}N_u, {\rm Lip}N_s \leqslant \gamma$, where $N_j=\pi_j\circ N$ and $\pi_j:X\rightarrow X_j$ is the projection of $X$ in $X_j$, $j=u,s$.

Our goal is to find $W^{\rm u}_{\rho}$ in the form $\{\xi+\theta(\xi)\colon \xi\in X_u\}$, for some Lipschitz map $\theta:X_u\rightarrow X_{s}$ with $\theta(0)=0$ and ${\rm Lip} \ \theta \leqslant 1$. The condition $SW^{\rm u}_{\rho}\subset W^{\rm u}_{\rho}$ implies that, for all $\xi \in X_u$, there exists $\hat{\xi}\in X_u$ with $S(\xi+\theta(\xi))=\hat{\xi}+\theta(\hat{\xi})$; that is,
\begin{equation}\label{ptofixoU}
\left\{
\begin{split}
&\hat{\xi}=L_u\xi+N_u(\xi+\theta(\xi))\\
&\theta_{\ast}(\hat{\xi})=L_s\theta(\xi)+N_s(\xi+\theta(\xi))
\end{split}\right.
\end{equation}
with $\theta_{\ast}=\theta$, where $\theta$ is the fixed point of the map $\theta\mapsto \theta_{\ast}$ defined by \eqref{ptofixoU}. We will show that this map has indeed a fixed point $\theta$ and then that $W^{\rm u}_{\rho}\doteq \{\xi+\theta(\xi)\colon \xi\in X_u\}$ satisfies the conditions given in this theorem. Let $\theta\colon X_u\rightarrow X_{s}$ be a Lipschitz map with $\theta(0)=0$ and ${\rm Lip} \ \theta\leqslant 1$. Then the map $\xi\mapsto -L_u^{-1}N_u(\xi+\theta(\xi))$ is a contraction in $X_u$, since
\begin{equation*}
\begin{split}
\|-L_u^{-1}N_u(\xi_1+\theta(\xi_1))+L_u^{-1}N_u(\xi_2+\theta(\xi_2))\|&\leqslant a^{-1}\gamma\left[\|\xi_1-\xi_2\|+\|\theta(\xi_1)-\theta(\xi_2)\|\right]\\
&\leqslant 2a^{-1}\gamma\|\xi_1-\xi_2\|,
\end{split}
\end{equation*}
and $2a^{-1}\gamma<1$. Therefore, for each $\hat{\xi}\in X_u$, there exists a unique $\xi \in X_u$ satisfying the first equation in \eqref{ptofixoU}, since for each $\hat{\xi}\in X_u$, the map
\begin{equation}\label{ptofixoXi}
X_u\ni \xi \mapsto L_u^{-1}\hat{\xi}-L_u^{-1}N_u(\xi+\theta(\xi))\in X_u
\end{equation}
is a contraction, hence is has a unique fixed point $\xi \in X_u$. Thus \eqref{ptofixoU} defines a map $\theta_{\ast}:X_u\rightarrow X_{s}$, and clearly $\theta_{\ast}(0)=0$. Given $\hat{\xi_1},\hat{\xi_2}\in X_u$, with $\xi_1, \xi_2\in X_u$ the correspondent fixed points of the map given in \eqref{ptofixoXi}, we have that
\begin{equation*}
\begin{split}
\|\hat{\xi_1}-\hat{\xi_2}\|&=\|L_u\xi+N_u(\xi_1+\theta(\xi_1))-L_u\xi_2-N_u(\xi_2+\theta(\xi_2))\|\\
&\geqslant \|L_u(\xi_1-\xi_2)\|-\gamma\big[\|\xi_1-\xi_2\|+\|\theta(\xi_1)-\theta(\xi_2)\|\big]\\
&\geqslant (a-2\gamma)\|\xi_1-\xi_2\|.
\end{split}
\end{equation*}

Also
\begin{equation}
\label{eq:LipTheta}
\begin{split}
\|\theta_{\ast}(\hat{\xi_1})-\theta_{\ast}(\hat{\xi_2})\|&=\|L_s\theta(\xi_1)+N_s(\xi_1+\theta(\xi_1))-L_s\theta(\xi_2)-N_s(\xi_2+\theta(\xi_2))\| \\ &\leqslant b\|\theta(\xi_1)-\theta(\xi_2)\|+\gamma\big[\|\xi_1-\xi_2\|+\|\theta(\xi_1)-\theta(\xi_2)\|\big] \\
& \leqslant (b+2\gamma)\|\xi_1-\xi_2\|,
\end{split}
\end{equation}
 therefore $\|\theta_{\ast}(\hat{\xi_1})-\theta_{\ast}(\hat{\xi_2})\|\leqslant  \frac{b+2\gamma}{a-2\gamma}\|\hat{\xi_1}-\hat{\xi_2}\|$; that is, ${\rm Lip} \ \theta_{\ast}\leqslant \frac{b+2\gamma}{a-2\gamma}<1$. The set $\mathcal{F}$ constituted of all Lipschitz maps $\theta\colon X_u\rightarrow X_{s}$ with $\theta(0)=0$, ${\rm Lip} \ \theta\leqslant 1$ with distance given by
 $$
 {\rm d}(\theta, \tau)=\sup_{\stackrel{\xi\in X_u}{\xi\neq 0}}\frac{\|\theta(\xi)-\tau(\xi)\|}{\|\xi\|},
 $$
is a complete metric space, and the map $\theta \mapsto \theta_{\ast}$ takes $\mathcal{F}$ into itself. We aim to prove now that this map is also a contraction in $\mathcal{F}$; and to this end, let $\theta, \tau$ be two maps in $\mathcal{F}$ and $\hat{\xi}\in X_u$. Define $\xi_{\theta},\xi_{\tau} \in X_u$ by $\hat{\xi}=L_u\xi_{\theta}+N_u(\xi_{\theta}+\theta(\xi_{\theta}))=L_u\xi_{\tau}+N_u(\xi_{\tau}+\tau(\xi_{\tau}))$. We have that $\|\hat{\xi}\|\geqslant (a-2\gamma)\|\xi_{\theta}\|$ and
\begin{equation*}
\begin{split}
a\|\xi_{\theta}-\xi_{\tau}\|&\leqslant \|N_u(\xi_{\theta}+\theta(\xi_{\theta}))-N_u(\xi_{\tau}+\tau(\xi_{\tau}))\| \leqslant \gamma\big[\|\xi_{\theta}-\xi_{\tau}\|+\|\theta(\xi_{\theta})-\tau(\xi_{\tau})\|\big]\\
& \leqslant \gamma\big[\|\xi_{\theta}-\xi_{\tau}\|+\|\theta(\xi_{\theta})-\tau(\xi_{\theta})\|+\|\tau(\xi_{\theta})-\tau(\xi_{\tau})\|\big]\\
& \leqslant \gamma\big[\|\xi_{\theta}-\xi_{\tau}\|+{\rm d}(\theta,\tau)\|\xi_{\theta}\|+\|\xi_{\theta}-\xi_{\tau}\|\big]\leqslant \gamma\big[2\|\xi_{\theta}-\xi_{\tau}\|+{\rm d}(\theta,\tau)\|\xi_{\theta}\|\big],
\end{split}
\end{equation*}
 thus
$$
\|\xi_{\theta}-\xi_{\tau}\|\leqslant \frac{\gamma {\rm d}(\theta,\tau)}{a-2\gamma}\|\xi_{\theta}\|\leqslant \frac{\gamma {\rm d}(\theta,\tau)}{(a-2\gamma)^2}\|\hat{\xi}\|.
$$

Moreover
\begin{equation*}
\begin{split}
\|\theta_{\ast}(\hat{\xi})-\tau_{\ast}(\hat{\xi})\|&=\|L_s\theta(\xi_{\theta})+N_s(\xi_{\theta}+\theta(\xi_{\theta}))-L_s\tau(\xi_{\tau})-N_s(\xi_{\tau}+\tau(\xi_{\tau}))\|\\
&\leqslant b\|\theta(\xi_{\theta})-\tau(\xi_{\tau})\|+\gamma\left[\|\xi_{\theta}-\xi_{\tau}\|+\|\theta(\xi_{\theta})-\tau(\xi_{\tau})\|\right]\\
&\leqslant (b+\gamma)\|\theta(\xi_{\theta})-\tau(\xi_{\tau})\| +\gamma\|\xi_{\theta}-\xi_{\tau}\|\\
& \leqslant (b+\gamma)\left[\|\theta(\xi_{\theta})-\tau(\xi_{\theta})\|+\|\tau(\xi_{\theta})-\tau(\xi_{\tau})\|\right]+\gamma\|\xi_{\theta}-\xi_{\tau}\|\\
& \leqslant (b+\gamma){\rm d}(\theta,\tau)\|\xi_{\theta}\|+(b+2\gamma)\|\xi_{\theta}-\xi_{\tau}\|\\
& \leqslant {\rm d}(\theta,\tau)\|\hat{\xi}\|\left[\frac{b+\gamma}{a-2\gamma}+\frac{\gamma}{a-2\gamma}\cdot\frac{b+2\gamma}{a-2\gamma}\right] \leqslant
 {\rm d}(\theta,\tau)\|\hat{\xi}\|\frac{b+2\gamma}{a-2\gamma},
\end{split}
\end{equation*}
which implies that ${\rm d}(\theta_{\ast},\tau_{\ast})\leqslant \frac{b+2\gamma}{a-2\gamma}{\rm d}(\theta,\tau)$  with  $\frac{b+2\gamma}{a-2\gamma}<1$, hence the map $\theta\mapsto \theta_{\ast}$ is a contraction in $\mathcal{F}$. Let $\theta$ be the fixed point of this map. Using equation \eqref{eq:LipTheta} for this fixed point, we obtain
$$
\|\theta(\hat{\xi}_1)-\theta(\hat{\xi}_2)\|\leqslant \frac{(b+\gamma){\rm Lip} \ \theta+\gamma}{a-2\gamma}\|\hat{\xi}_1-\hat{\xi}_2\|,
$$
therefor ${\rm Lip}(\theta) \leqslant \frac{(b+\gamma){\rm Lip} \ \theta+\gamma}{a-2\gamma}$, which implies that
\begin{equation}
\label{eq:LipTheta2}
{\rm Lip}( \theta) \leqslant \frac{\gamma}{a-b-3\gamma}<1.
\end{equation}

Now we define $W^{\rm u}_{\rho}=\{\xi+\theta(\xi)\colon \xi\in X_u\}$. We firstly claim that $SW^{\rm u}_{\rho}=W^{\rm u}_{\rho}$, and to show this, let $x=\xi+\theta(\xi)\in W^{\rm u}_{\rho}$. We have $Sx=\hat{\xi}+\hat{\eta}$, with $\hat{\xi}\in X_u$, $\hat{\eta} \in X_{s}$ , $\hat{\xi}=L_u\xi+N_u(\xi+\theta(\xi))$ and $\hat{\eta}=L_s\theta(\xi)+N_s(\xi+\theta(\xi))=\theta(\hat{\xi})$, by definition of $\theta$. Thus $SW^{\rm u}_{\rho}\subset W^{\rm u}_{\rho}$. On the other hand, $\hat{\xi}=L_u\xi+N_u(\xi+\theta(\xi))$ and so $S(\xi+\theta(\xi))=\hat{\xi}+\theta(\hat{\xi})$, which proves that $W^{\rm u}_{\rho}\subset SW^{\rm u}_{\rho}$ and concludes the claim.

Now, if $x=\xi+\theta(\xi)$ and $z=\zeta+\theta(\zeta)$ are points of $W^{\rm u}_{\rho}$, $Sx=\hat{\xi}+\theta(\hat{\xi})$ and $Sz=\hat{\zeta}+\theta(\hat{\zeta})$ then $\|x-z\|=\max\{\|\xi-\zeta\|,\|\theta(\xi)-\theta(\zeta)\|\}=\|\xi-\zeta\|$, since $\|\theta(\xi)-\theta(\zeta)\|\leqslant\|\xi-\zeta\|$. Also
$$
\|Sx-Sz\|=\|\hat{\xi}-\hat{\zeta}\|\geqslant (a-2\gamma)\|\xi-\zeta\|=(a-2\gamma)\|x-z\|,
$$
and $a-2\gamma>\rho$, which proves (iii) for $W^{\rm u}_\rho$; that is, ${\rm Lip}(S|_{W^{\rm u}_{\rho}}^{-1})\leqslant (a-2\gamma)^{-1}<\rho^{-1}$.

Given $x\in W^{\rm u}_{\rho}$, there exist $\{x_j\}_{j<0}$ in $W^{\rm u}_{\rho}$ such that $S^{|j|}x_j=x$ and $S(x_{j+1})=x_j$, for all $j<0$ with $x_0\doteq x$ (since $SW^{\rm u}_{\rho}=W^{\rm u}_{\rho}$). Also $
\|x_{j+1}\|=\|S(x_j)-S(0)\|\geqslant (a-2\gamma)\|x_j\|$,  and thus the map $j\mapsto \frac{\|x_j\|}{(a-2\gamma)^j}$ is increasing, since
$$
\frac{\|x_j\|}{(a-2\gamma)^j}\leqslant \frac{\|x_{j+1}\|}{(a-2\gamma)^{j+1}}, \hbox{ for all }j<0,
$$
which implies that $\frac{\|x_j\|}{(a-2\gamma)^j}\leqslant \|x\|$, and so $\|x_j\|\leqslant \|x\|(a-2\gamma)^j=o(\rho^j)$ as $j\to -\infty$. Therefore
\begin{equation*}
\begin{split}
W^{\rm u}_\rho\subset\{x\in X\colon \hbox{ there exists } &\{x_j\}_{j\leqslant 0} \hbox{ with }x_0=x, \\
				&x_{j+1}=S(x_j), \hbox{ and } \sup_{j\leqslant 0}\rho^{-j}\|x_j\|<\infty\}
\end{split}
\end{equation*}

\underline{Note:} when $\sup_{j\leqslant 0}\rho^{-j}\|x_j\|<\infty$, we say $\|x_j\|=O(\rho^j)$ as $j\to -\infty$.

Now let $x\in X$, $\{x_j\}_{j\leqslant 0}$, with $x_0=x$, such that $T(x_j)=x_{j+1}$, for all $j<0$ and $\|x_j\|=O(\rho^j)$ as $j\to -\infty$. We write $x_j=\xi_j+\theta(\xi_j)+\eta_j$, where $\xi_j\in X_u, \eta_j\in X_{s}$ and since $T(x_j)=x_{j+1}$, we have that $L_u\xi_j+N_u(\xi_j+\theta(\xi_j)+\eta_j)+L_s[\theta(\xi_j)+\eta_j]+N_s(\xi_j+\theta(\xi_j)+\eta_j) = \xi_{j+1}+\theta(\xi_{j+1})+\eta_{j+1}$; that is,
\begin{equation*}
\left\{
\begin{split}
&\xi_{j+1}=L_u\xi_j+N_u(\xi_j+\theta(\xi_j)+\eta_j)\\
&\theta(\xi_{j+1})+\eta_{j+1}=L_s(\theta(\xi_j)+\eta_j)+N_s(\xi_j+\theta(\xi_j)+\eta_j)
\end{split}\right.
\end{equation*}
and if $\hat{\xi_j}\doteq L_u\xi_j+N_u(\xi_j+\theta(\xi_j))$ then $\theta(\hat{\xi_j})=L_s\theta(\xi_j)+N_s(\xi_j+\theta(\xi_j))$. Also $\|\hat{\xi_j}-\xi_{j+1}\|\leqslant \gamma\|\eta_j\|$ and
\begin{equation*}
\begin{split}
\|\eta_{j+1}\|-\|\xi_{j+1}-\hat{\xi_j}\|&\leqslant \|\theta(\xi_{j+1})+\eta_{j+1}-\theta(\hat{\xi_j})\|\leqslant\\
&\leqslant b\|\eta_j\|+\gamma\|\eta_j\|=(b+\gamma)\|\eta_j\|,
\end{split}
\end{equation*}
which implies that $\|\eta_{j+1}\|\leqslant (b+2\gamma)\|\eta_j\|$.

Now $\|\xi_j\|\leqslant \|x_j\|=\max\{\|\xi_j\|,\|\theta(\xi_j)+\eta_j\|\}=O(\rho^j)$, thus $\|\xi_j\|=O(\rho^j)$, and therefore $\|\eta_j\|=\|x_j-\xi_j-\theta(\xi_j)\|\leqslant \|x_j\|+2\|\xi_j\|=O(\rho^j)$ as $j\to -\infty$. But $\|\eta_j\|\geqslant (b+2\gamma)^j\|\eta_0\|$, for all $j\leqslant 0$, so there exist $c>0$ s.t. $\|\eta_0\|\leqslant \frac{\|\eta_j\|}{(b+2\gamma)^j}\leqslant c(\frac{\rho}{b+2\gamma})^j\rightarrow 0$, as $j\to -\infty$, which shows that $\eta_0=0$ and $x=x_0=\xi_0+\theta(\xi_0)\in W^{\rm u}_{\rho}$, which concludes the proof of (v) for $W^{\rm u}_\rho$.

For the case $W^{\rm s}_{\rho}=\{\sigma(\eta)+\eta\colon \eta\in X_{s}\}$, $\sigma\colon X_{s}\rightarrow X_u$, the condition $SW^{\rm s}_{\rho}\subset W^{\rm s}_{\rho}$ implies that for all $\eta \in X_{s}$ there exists $\hat{\eta}\in X_{s}$ with
\begin{equation}\label{ptofixoS}
\left\{
\begin{split}
&\sigma(\hat{\eta})=L_u\sigma_{\ast}(\eta)+N_u(\sigma(\eta)+\eta)\\
&\hat{\eta}=L_s\eta+N_s(\sigma(\eta)+\eta)
\end{split}\right.
\end{equation}
with $\sigma_{\ast}=\sigma$ and the remainder of the proof is completely analogous.

Finally, we will show that $W^{\rm s}_{\rho}\cap W^{\rm u}_{\rho}=\{0\}$. Let $x=\xi+\eta\in W^{\rm s}_{\rho}\cap W^{\rm u}_{\rho}$, then $\xi=\sigma(\eta)$ and $\eta=\theta(\xi)$. Thus $\xi=\sigma(\theta(\xi))$ and $\|\xi\|\leqslant {\rm Lip} \ \sigma\cdot {\rm Lip} \ \theta\|\xi\|$, therefore $\xi=0$ and implies that $\eta=0$, consequently $x=0$.
\end{proof}

\begin{remark} It is clear from equation \eqref{eq:LipTheta2} that ${\rm Lip}(\theta)\to 0$ as $\gamma\to 0$.
\end{remark}

We see that this result is stated with the hypothesis that $N$ is a Lipschitz map with a small Lipschitz constant $\gamma$ in the {\it whole} space $X$, but it is important to consider the case when $N$ is only defined in a small neighborhood of $0$, and to this end, note that if $N\colon \overline{B_r^X(0)}\to X$ is a Lipschitz map with $N(0)=0$ and ${\rm Lip}(N)\leqslant \gamma$, then the function $\tilde{N}\colon X\to X$ defined by
\begin{equation}\label{eq:ExtensionN}
\tilde{N}(x)=\left\{
\begin{split}
&N(x),\ \hbox{ if }\|x\|\leqslant r\\
&N\left(\frac{r}{\|x\|}x\right), \ \hbox{ if } \|x\|>r
\end{split}
\right.
\end{equation}
is an extension of $N$ to $X$, Lipschitz in the whole space $X$, with
$$
{\rm Lip}(\tilde{N})\leqslant 2{\rm Lip}(N) \hbox{ and } \sup_{x\in X}\|\tilde{N}(x)\|\leqslant \gamma r.
$$

Hence, to obtain invariant manifolds in a neighborhood of $0$ for $N$, we apply Theorem \ref{teoUS} for $\tilde{N}$ and obtain the following result.

\begin{theorem}[Local invariant manifolds] \label{teo:LIM}
Let $X$ be a Banach space, $L:X\rightarrow X$ a bounded linear operator, $\rho>0$ such that $\sigma(L) \cap \{\lambda \in \mathbb{C}\colon |\lambda|=\rho\}=\varnothing$. Then, there exist $\gamma >0$, a neighborhood $U$ of $0$ in $X$, a decomposition $X=X_u\oplus X_{s}$ and an equivalent norm in $X$ such that if $N\colon U\rightarrow X$ is Lipschitz continuous and satisfies $N(0)=0$ and ${\rm Lip}(N)\leqslant \gamma$, there exist sets $W^{\rm s}_{loc,\rho}$, $W^{\rm u}_{loc,\rho}$ in $X$ and neighborhoods $V_u,V_s$ of $0$ in $X_u,X_s$, respectively, such that for $S=L+N$, we have the following:
\begin{enumerate}
\item[{\bf (a)}] $W^{\rm u}_{loc,\rho}$ is a graph of a Lipschitz map over $V_u$, $W^{\rm s}_{loc,\rho}$ is a graph of a Lipschitz map over $V_s$.
\item[{\bf (b)}] $S(W^{\rm u}_{loc,\rho})\supset W^{\rm u}_{loc,\rho}$, $S(W^{\rm s}_{loc,\rho})\subset W^{\rm s}_{loc,\rho}\cap S(X)$ and the restriction $S|_{W^{\rm u}_{loc,\rho}}\colon W^{\rm u}_{loc,\rho} \to S(W^{\rm u}_{loc,\rho})$ is a homeomorphism.
\item[{\bf (c)}] $W^{\rm s}_{loc,\rho}\cap W^{\rm u}_{loc,\rho}=\{0\}$.
\item[{\bf (d)}] ${\rm Lip}\left(S|_{W^{\rm s}_{loc,\rho}}\right)<\rho$ and ${\rm Lip}\left(S|_{W^{\rm u}_{loc,\rho}}^{-1}\right)<\rho^{-1}$.
\item[{\bf (e)}] We have
\begin{equation*}
\begin{split}
&\quad {\bf 1.} \ W^{\rm s}_{loc,\rho}=\{x\in U\colon S^n(x)\in U \hbox{ for all }n\geqslant 0 \hbox{ and } \|S^nx\|=o(\rho^n) \hbox{ as }n\to \infty\},\\
&\quad {\bf 2.} \ W^{\rm u}_{loc,\rho}=\{x\in U\colon \hbox{ there exists } \{x_j\}_{j\leqslant 0}\subset U \hbox{ with }x_0=x, x_{j+1}=S(x_j) \hbox{ and } \sup_{j\leqslant 0}\rho^{-j}\|x_j\|<\infty\},
\end{split}
\end{equation*}
and for $\rho^{\ast}$ sufficiently close to $\rho$, $W^{\rm s}_{loc,\rho^{\ast}}=W^{\rm s}_{\rho}$ and $W^{\rm u}_{loc,\rho^{\ast}}=W^{\rm u}_{\rho}$;
\end{enumerate}

\end{theorem}

\section{$\mathfrak{L}$-hyperbolicity and isolated global solutions} \label{sec:LHP}

In this section we deal with {\it invariant manifolds} near some special fixed points, that we call {\it $\mathfrak{L}$-hyperbolic points} and also with the {\it isolation} of such points, in the sense of global solutions, that we will specify in details below.

\subsection{Weakly hyperbolic points and invariant manifolds}

\begin{definition}\label{w_hyperbolic_fixed_point}
Let $x^\ast$ be an equilibrium point for a map $T\in C(X)$. We say that $x^*$ is a \df{weakly hyperbolic point} if the map $S\colon X\to X$ defined by
\begin{equation}
\label{eq:S}
S(x)=T(x+x^\ast)-T(x^\ast)=T(x+x^\ast)-x^\ast, \hbox{ for }x\in X,
\end{equation}
has a decomposition $S=L+N$ satisfying the following conditions:
\begin{itemize}
\item[\bf (i)] $L\in \mathcal{L}(X)$ is such that $\sigma(L)\cap S^1=\varnothing$ and $0<b<1<a$ are as is \eqref{eq:AB} and
\item[\bf (ii)] there exists a neighborhood $U$ of $0$ in $X$ such that $N\colon U\to X$ has small Lipschitz constant with respect to $L$,
\item[\bf (iii)] the constant $\gamma={\rm Lip}(N)$ given by $(ii)$ is such that $0<\gamma \cdot \|(I-L)^{-1}\|\leqslant 1$.
\end{itemize}

Choosing $\delta>0$ such that $B_\delta^{X}(0)\subset U$, we say also that $x^*$ is a \df{weakly hyperbolic point with parameters} $\gamma,a,b,\delta$.
\end{definition}

We clearly have $N(0)=S(0)-L0 = T(x^\ast)-x^\ast=0$. Also, note that $(I-L)^{-1}x=(I-L_u)^{-1}x_u+(I-L_s)^{-1}x_{s}$ and then $\|(I-L)^{-1}\|\leqslant \frac{a}{a-1}+\frac{1}{1-b}$. Thus property (iii) holds true if $\gamma$ can be chosen such that $0<\gamma \leqslant \frac{(a-1)(1-b)}{a(2-b)-1}$.

Using Theorem \ref{teo:LIM} we obtain as an straightforward application the main result concerning Lipschitz invariant manifolds near a weakly hyperbolic point.

\begin{definition} 
Let $x^\ast\in X$ a hyperbolic fixed point for a map $T\in C(X)$. We define the \df{ local unstable manifold} of $x^\ast$ as  
$$W^{u}_{\rm loc}(x^\ast)=\{x^\ast+\xi+\theta_u(\xi)\colon \|\xi\|<r, \ \xi \in X_u\},$$ 
were $\theta_u$ and $X_u$ as given by Theorem \ref{teoUS}. The same way we define the local stable manifold.
\end{definition}

\begin{theorem} \label{cor:US}
Let $X$ be a Banach space, $T\in C(X)$ with a weakly hyperbolic point $x^\ast$.  Then there exist a decomposition $X=X_u\oplus X_s$, an $\epsilon_0>0$, neighborhoods $U$ of $x^\ast$ in $X$ and $V_u,V_s$ of $0$ in $X_u,X_s$, respectively, and sets $W^{\rm u}_{loc}(x^\ast)$, $W^{\rm s}_{loc}(x^\ast)$ such that
\begin{enumerate}
\item[{\bf (a)}] $W^{\rm u}_{loc}(x^\ast)$ is a graph of a Lipschitz map over $x^\ast+V_u$, $W^{\rm s}_{loc}(x^\ast)$ is a graph of a Lipschitz map over $x^\ast+V_s$.
\item[{\bf (b)}] $T(W^{\rm u}_{loc}(x^\ast))\supset W^{\rm u}_{loc,\rho}(x^\ast)$, $T(W^{\rm s}_{loc}(x^\ast)\subset W^{\rm s}_{loc,\rho}(x^\ast)\cap T(X)$ and the following restriction $T|_{W^{\rm u}_{loc,\rho}(x^\ast)}\colon W^{\rm u}_{loc,\rho}(x^\ast) \to T(W^{\rm u}_{loc,\rho}(x^\ast))$ is a homeomorphism.
\item[{\bf (c)}] $W^{\rm s}_{loc}(x^\ast)\cap W^{\rm u}_{loc}(x^\ast)=\{x^\ast\}$.
\item[{\bf (d)}] ${\rm Lip}\left(T|_{W^{\rm s}_{loc}(x^\ast)}\right)<1$ and ${\rm Lip}\left(T|_{W^{\rm u}_{loc}(x^\ast)}^{-1}\right)<1$.
\item[{\bf (e)}] We have
\begin{equation*}
\begin{split}
&\quad {\bf 1.} \ W^{\rm s}_{loc}(x^\ast)=\{x\in U\colon T^n(x)\in U \hbox{ for all }n\geqslant 0 \hbox{ and } \|T^nx-x^\ast\|=o((1-\epsilon_0)^n) \hbox{ as }n\to \infty\},\\
&\quad {\bf 2.} \ W^{\rm u}_{loc}(x^\ast)=\{x\in U\colon \hbox{ there exists } \{x_j\}_{j\leqslant 0}\subset U \hbox{ with }x_0=x, x_{j+1}=T(x_j) \hbox{ and }\\
& \hspace{10cm} \sup_{j\leqslant 0}(1+\epsilon_0)^{-j}\|x_j-x^\ast\|<\infty\}.
\end{split}
\end{equation*}
\end{enumerate}
\end{theorem}

We say that $W^{\rm u}_{loc}(x^{\ast})$ is the \df{local unstable manifold} of $x^{\ast}$ and $W^{\rm s}_{loc}(x^{\ast})$ is the \df{local stable manifold} of $x^\ast$. Also, the \df{dimension} ${\rm dim} \ W^{\rm j}_{loc}(x^{\ast})$ is defined as the dimension of $X_j$, for $j=u,s$.

\subsection{$\mathfrak{L}$-hyperbolic points} We define in this subsection the concept of {\it $\mathfrak{L}$-hyperbolic points}, which is crucial for our study. 

\begin{definition}
Let $X,Y$ be Banach spaces, $U\subset X$ and $V\subset Y$. We say that a map $g\colon U\to V$ is \df{bi-Lipschitz} if $g\colon U\to V$ is invertible with inverse $g^{-1}\colon V\to U$ and both $g$ and $g^{-1}$ are Lipschitz continuous maps.
\end{definition}

\begin{proposition}\label{hu3}
Let $S\in C(X)$ with $0$ as a weakly hyperbolic point and decomposition $S=L+N$. Let $X=X_u\oplus X_s$ be the decomposition given in Definition \ref{w_hyperbolic_fixed_point}. Then there exist $\gamma_1=\gamma_1(L)>0$ such that if $\tn{Lip}(N)<\gamma_1$, then there exists neighborhoods $U,V_u,V_s$ of $0$ in $X,X_u,X_s$, respectively, and a bi-Lipschitz map $g\colon X\to X$ which satisfies
\begin{itemize}
\item[\bf (a)] $g(0)=0$;

\item[\bf (b)] if $S_1=g^{-1}\circ T \circ g$, then $0$ is a weakly hyperbolic point of $S_1$ with decomposition $S=L+N_1$;

\item[\bf (c)] if $W^{\rm u,1}_{loc}(0)$ and $W^{\rm s,1}_{loc}(0)$ are the invariant manifolds given in Corollary \ref{cor:US} for $S_1$, then
$$
W^{\rm u,1}_{loc}(0)\cap U = V_u \hbox{ and }W^{\rm s,1}_{loc}(0) \cap U = V_s;
$$
in particular, $N_{1,u}(x_s)=N_{1,s}(x_u)=0$ for all $x_u\in V_u$ and $x_s\in V_s$.
\end{itemize}
\end{proposition}
\begin{proof} This proof is fairly technical, and so in order to give a clear outline of the theory, it is present in Section \ref{app:ProofHu3} of Appendix \ref{app:A}.
\end{proof}

\begin{definition}\label{hyperbolic_fixed_point}
Let $x^\ast$ be a weakly hyperbolic point for a map $T\in C(X)$ with decomposition $S=L+N$ as in Definition \ref{w_hyperbolic_fixed_point}. We say that $x^*$ is a \df{$\mathfrak{L}$-hyperbolic point} if $\tn{Lip}(N)<\gamma_1$ with $\gamma_1$ given by Proposition \ref{hu3}.
\end{definition}

\subsection{Isolation of $\mathfrak{L}$-hyperbolic points}  We begin this subsection with the definition of {\it isolated solution}.

\begin{definition}
Let $\xi$ be a global solution of a map $T\in C(X)$. We say that $\xi$ is {\bf isolated} if there exists a neighborhood $U$ of $\xi(\mathbb{Z})$ in $X$ such that $\xi$ is the only global solution of $T$ with $\xi(\mathbb{Z})\subset U$.
\end{definition}

\begin{proposition}\label{hu2}
Let $S\in C(X)$ with $0$ as a $\mathfrak{L}$-hyperbolic point. Then there exists $\delta>0$ such that $0$ is the unique global solution of $S$ in $B_\delta^X(0)$; that is, $0$ is isolated.

\end{proposition}
\begin{proof} By Proposition \ref{hu3}, let $g\colon X\to X$ be a bi-Lipschitz map and $S_1=g^{-1}\circ T\circ g = L +N_1$ with $0$ as an weakly hyperbolic point. Also, we have $N_{1,u}(x_s)=N_{1,s}(x_u)=0$ for all $x_u\in X_u$ and $x_s\in X_s$. Hence, for $x=x_u+x_s \in V_u\oplus V_s$ and $b+\gamma<1<a-\gamma$ as in Definition \ref{w_hyperbolic_fixed_point}, we have
\begin{align*}
&\|\pi_u S_1(x)\|=\|L_ux_u+N_{1,u}(x_u+x_{s})-N_{1,u}(x_{s})\|\geqslant (a-\gamma)\|x_u\|, \hbox{ and }\\
&\|\pi_sS_1(x)\|= \|L_sx_{s}+N_{1,s}(x_u+x_{s})-N_{1,s}(x_u)\|\leqslant (b+\gamma)\|x_{s}\|,
\end{align*}

Since $g^{-1}$ is Lipschitz continuous in $X$ and $g^{-1}(0)=0$, there exists $M\geqslant 1$ such that $\|g^{-1}(x)\|\leqslant M\|x\|$ for all $x\in X$. Let $\delta >0$ such that $B_{M\delta}^X(0) \subset V_u\oplus V_s$. If $\xi:\Z\to X$ is a global solution of $S_1$ in $B_{M\delta}^X(0)$ and there exists $m\in \Z$ such that $\pi_u\xi(m)\neq 0$, the fact that $\delta \geqslant \|\pi_uS_1^n\xi(m))_u\|\geqslant (a-\gamma)^n\|\pi_u\xi(m)\|$ gives us a contradiction, making $n\to \infty$. Thus, $\pi_u\xi(m)=0$ for all $m\in \Z$. On the other hand, since
\[
\|\pi_s\xi(m)\|=\|\pi_sS_1^n\xi(m-n)\|\leqslant (b+\gamma)^n\|\pi_s\xi(m-n)\|< M \delta(b+\gamma)^n,
\]
making $n\to \infty$ we obtain $\pi_s\xi(m)=0$ for each $m\in \mathbb{Z}$. Hence $0$ is the unique global solution of $S_1$ in $B_\delta^X(0)$.

Now, if $\phi\colon \mathbb{Z}\to X$ is a global solution of $S$ in $B_{\delta}^X(0)$ then $\xi = g^{-1}\circ \phi \colon \mathbb{Z}\to X$ is a global solution of $S_1$ in $B_{M\delta}^X(0)$. Hence $\xi \equiv 0$ which in turn implies that $\phi\equiv 0$.
\end{proof}


%
%

Now, the main theorem of this subsection follows easily.

\begin{theorem}\label{hu1}
Let $T\in C(X)$ with an $\mathfrak{L}$-hyperbolic point $x^\ast$. Then $x^*$ is isolated.
\end{theorem}
\begin{proof}
It follows directly from Proposition \ref{hu3}.
\end{proof}

With this result, we are able to understand what happens near a $\mathfrak{L}$-hyperbolic point $x^*$, as follows.

\begin{definition}
We say that $T\in C(X)$ is \df{asymptotically compact} if the discrete semigroup $\{T^n\colon n\in \mathbb{N}\}$ is asymptotically compact; that is, if given sequences $n_k\to \infty$ in $\mathbb{N}$ and $\{x_k\}$ bounded in $X$ such that $\{T^{n_k}(x_k)\}$ is bounded in $X$, then $\{T^{n_k}(x_k)\}$ has a convergent subsequence.
\end{definition}

\begin{theorem}\label{hu1*}
Let $T\in C(X)$ an asymptotically compact map with $x^\ast$ as an $\mathfrak{L}$-hyperbolic equilibrium. Then there exists $\delta>0$ such that if $\xi:\Z\to X$ is a global solution of $\ T$ with $\xi(m)\in B_{\delta}^X(x^*)$ for all $m\geqslant 0$ $(m\leqslant 0)$, then $\xi(m)\to x^*$ as $m\to \infty$ $(m\to -\infty)$.
\end{theorem}
\begin{proof}
Let $\delta>0$ be as in Theorem \ref{hu1} and take $\delta_1=\frac{\delta}{2}>0$. Let $\xi$ be a global solution of $T$ such that $\xi(m)\in B_{\delta_1}^X(x^\ast)$ for each $m\geqslant 0$ and assume that $\xi(m)$ does not converge to $x^\ast$ as $m\to \infty$. Then there exist $\epsilon_0>0$ and a sequence $n_k\to \infty$ such that
$$
\|T^{n_k}(\xi(0))-x^\ast\| = \|\xi(n_k)-x^\ast\| \geqslant \epsilon_0, \hbox{ for each }k\in \mathbb{N}.
$$

From the asymptotically compactness of $T$, we can extract a subsequence if necessary of $\{T^{n_k}(\xi(0))\}$ which converges to a point $z\in \overline{B_{\delta_1}^X(x^\ast)}$ such that $\|z-x^\ast\|\geqslant \epsilon_0$. Define $\phi(m)=T^mz$ for each $m\geqslant 0$; clearly $\phi(m)\in \overline{B_{\delta_1}^X(x^\ast)}$ since for each $m\geqslant 0$, $\phi(m)=\lim_{k\to \infty}\xi(n_k+m)$.

The sequence $\{T^{n_k-1}(\xi(0))\}\subset B_{\delta_1}^X(x^\ast)$ also has a convergent subsequence, which we denote the same, to a point $z_{-1}\in \overline{B_{\delta_1}^X(x^\ast)}$. Clearly $Tz_{-1}=z$ and we define $\phi(-1)=z_{-1}$.

Continuing this process we define points $\phi(-m)$ in $\overline{B_{\delta_1}^X(x^\ast)}$, for $m>0$, such that $T\phi(-m)=\phi(-m+1)$. Thus $\phi$ is a global solution of $T$ in $B_{\delta}^X(x^\ast)$, and hence it must be the equilibrium solution $x^\ast$ by Theorem \ref{hu1}. Therefore $0=\|x^\ast-x^\ast\|=\|\phi(0)-x^\ast\|=\|z-x^\ast\|\geqslant \epsilon_0$, which gives us a contradiction.

The proof for the other case is analogous.
\end{proof}

\section{Autonomous perturbations of $\mathfrak{L}-$hyperbolic points} \label{sec:APHP}

In this section, we study the permanence of $\mathfrak{L}-$hyperbolic points under small Lipschitz continuous autonomous perturbations. To begin, we will precisely define what do we mean by {\it small perturbation}. 

\begin{definition}  \label{def:norms}
Let $X,Y$ be Banach spaces,  $U$ be a subset of $X$ and $T\colon U \to Y$. Define
\begin{equation*}
\|T\|_{U,Lip} = \sup_{\begin{smallmatrix}x,y\in U\\x\neq y\end{smallmatrix}}\frac{\|T(x)-T(y)\|}{\|x-y\|},
 \qquad \|T\|_{U,\infty}=\sup_{x\in U}\|T(x)\|,
 \end{equation*}
and also
\begin{equation} \label{eq:NormaT}
 \norma{T}_U =  \max\{\|T\|_{U,\infty},\|T\|_{U,Lip}\}.
\end{equation}

We say that a map $\tilde{T}\colon U\to Y$ is \df{$\epsilon$-Lipschitz close} to $T$ in $U$ if
\begin{equation*}
\norma{\tilde{T}-T}_U \leqslant \epsilon.
\end{equation*}
\end{definition}

\begin{proposition}\label{teoconv}
Let $S\in C(X)$ be such that $0$ is a weakly hyperbolic point of $S$ and assume that $S=L+N\colon U\to X$ is as in Definition \ref{w_hyperbolic_fixed_point}. Let $\delta >0$ be such that $\overline{B_\delta^X(0)} \subset U$ and choose $\epsilon >0$, $0<\delta_1<\delta$ such that
\[
(\epsilon +\gamma\delta_1) \|(I-L)^{-1}\| \leqslant \delta_1, (\epsilon +\gamma) \|(I-L)^{-1}\|<1 \hbox{ and } b+2(\epsilon+\gamma)<1<a-2(\epsilon+\gamma).
\]

If $S_1\in C(X)$ is $\epsilon$-close to $S$ then $S_1$ has a unique weakly hyperbolic point  $x_1^*\in  \overline{B_{\delta_1}^X(0)}$.
\end{proposition}

\begin{proof}
Let $\phi \colon \overline{B_{\delta_1}^X(0)}\to X$ be given by $\phi(x)=(I-L)^{-1}(S_1(x)-S(x)+N(x))$ and note that $S_1(x)=x$ iff $\phi(x)=x$. Since $N(0)=0$ we have
$$
\|\phi(x)\|\leqslant \|(I-L)^{-1}\|\cdot \|S_1(x)-S(x)+N(x)-N(0)\|\leqslant \|(I-L)^{-1}\|(\epsilon+\gamma\delta_1)\leqslant \delta_1;
$$
that is, $\phi(\overline{B_{\delta_1}^X(0)})\subset \overline{B_{\delta_1}^X(0)}$. Moreover
\begin{align}
\|\phi(x)-\phi(y)\|&\leqslant \|(I-L)^{-1}\| \cdot \|[(S_1-S)(x) - (S_1-S)(y)+N(x)-N(y)\| \nonumber \\
&\leqslant \|(I-L)^{-1}\|(\e+\gamma)\cdot \|x-y\|, \nonumber
\end{align}
for all $x,y\in B_{\delta_1}^X(0)$. Thus $\phi \colon \overline{B_{\delta_1}^X(0)}\to \overline{B_{\delta_1}^X(0)}$ is a contraction and possesses a unique fixed point $x_1^*$ in $\overline{B_{\delta_1}^X(0)}$.

Now, let $\delta_2>0$ be such that $B_{\delta_2}^X(x^\ast_1)\subset B_\delta^X(0)$ and consider the map $R(x)=S_1(x+x^\ast_1)-x^\ast_1$, for $x$ in $B_{\delta_2}^X(0)$. We have
\begin{align}
R(x) &= S_1(x+x^\ast_1)-S(x+x^\ast_1) + S(x+x^\ast_1)-x^\ast_1 = Lx+N_1(x),\nonumber
\end{align}
where $N\colon B_{\delta_2}^X(0)\to X$ is given by
$$
N_1(x)=N(x+x_1^*)+(S_1-S)(x+x_1^*)+L(x_1^*)-x_1^*.
$$

Clearly $N_1(0)=0$ and ${\rm Lip}(N_1) = \gamma + \epsilon$. Hence all the conditions of Definition \ref{w_hyperbolic_fixed_point} are satisfied, and $x^\ast_1$ is an weakly hyperbolic point of $S_1$.
\end{proof}


\begin{corollary}\label{convvarlip}
Let $T_0\in C(X)$ and $x_0^\ast$ an weakly hyperbolic point for $T_0$ and fix $U=B_{\delta}^X(x_0^\ast)$, for some $\delta>0$ sufficiently small. Assume that for each $\eta\in (0,1]$ we have a map $T_\eta\in C(X)$ with a set  $\mathcal{E}_\eta$ of equilibrium points, and suppose that $\norma{T_{\eta}-T_0}_{U}\to 0$ as $\eta\to 0^+$. Then there exist $\eta_0>0$ and weakly hyperbolic points $x^*_\eta\in \mathcal{E}_\eta$ of $T_\eta$, for each $0\leqslant \eta\leqslant \eta_0$, such that $x^*_\eta\to x^*_0$ as $\eta\to 0^+$. {\color{black}In particular, if $\mathcal{E}_0$ is compact and $U=\mathcal{O}_\delta(\mathcal{E}_0)$, then
$\{\mathcal{E}_\eta\}_{\eta\in[0,1]}$ is lower semicontinuous at $\eta=0$.

Moreover, if $\cup_{0\leqslant \eta\leqslant \eta_1}\mathcal{E}_\eta$ is precompact
for some $\eta_1>0$, then $\{\mathcal{E}_\eta\}_{\eta\in[0,1]}$ is upper semicontinuous
at $\eta=0$. Consequently, if $\mathcal{E}_0$ has only $\mathfrak{L}$-hyperbolic
equilibrium points, then there exist $\eta_1>0$ such that $\mathcal{E}_\eta$ has the
same number of elements of the $\mathcal{E}_0$ for all $0\leqslant \eta\leqslant \eta_1$.
In other words, there exists $\eta_1>0$ and $p\in \N$ such that $\mathcal{E}_\eta=
\{x_{1,\eta}^*,\hdots,x_{p,\eta}^{*}\}$ has only $\mathfrak{L}$-hyperbolic fixed points
for all $\eta\in [0,\eta_1]$ and
\[
\max_{i=1,\hdots,p}\|x_{i,\eta}^*-x_{i,0}^*\|\stackrel{\eta\to 0}{\longrightarrow}0.
\]}
\end{corollary}
\begin{proof}
From Proposition \ref{teoconv}, for each $T_\eta$ we can choose $\delta_\eta>0$ such that in $B_{\delta_\eta}^X(x_0^\ast)$ there exists a unique equilibrium point $x_\eta^\ast$, and $x_\eta^\ast$ is a weakly hyperbolic point for $T_\eta$. Choosing $0<\delta_\eta\leqslant \eta$ for each $\eta \in (0,1]$ we have $\|x_\eta^\ast-x_0^\ast\|\to 0$ as $\eta\to 0^+$.
\end{proof}

\begin{remark}\label{rem:UnifConstants}
We note that, from Proposition \ref{teoconv}, if $\gamma,a,b>0$ and $U$ are given as in Definition \ref{w_hyperbolic_fixed_point} for the weakly hyperbolic point $x^\ast_0$ of $T_0$ and $\delta>0$ is such that $\overline{B_\delta^X(0)}\subset U$, there exists $\eta_0>0$ and constants $\tilde{\gamma},\tilde{a},\tilde{b}>0$ and a neighborhood $\tilde{U}\subset U$ with $\gamma<\tilde{\gamma}$, $\tilde{b}<b<1<\tilde{a}<a$ such that they fullfill the conditions of Definition \ref{w_hyperbolic_fixed_point} for $x^\ast_\eta$, for all $0\leqslant \eta\leqslant \eta_0$.
\end{remark}

With these previous results, we are able to prove the continuity at $\eta=0$ of the family $\{\mathcal{E}_\eta\}_{\eta\in[0,1]}$ of equilibria for the maps $T_\eta$, assuming the {\it Lipschitz convergence} of $T_\eta$ to $T_0$; that is, the convergence of $T_\eta$ to $T_0$ in the norm $\norma{\cdot}_V$, for some suitable open set $V$.

\begin{theorem} \label{teo:CHP}
With the conditions of Proposition \ref{teoconv} assume that $\mathcal{E}_0$ is finite, $\cup_{\eta\in [0,1]}\mathcal{E}_\eta$ is precompact in $X$ and that $\norma{T_\eta-T_0}_V\to 0$ as $\eta \to 0^+$ for some neighborhood $V$ of $\overline{\cup_{\eta\in[0,1]}\mathcal{E}_\eta}$. Then there exists $\eta_0>0$ such that $\mathcal{E}_\eta$ is finite and possesses the same number of elements as $\mathcal{E}_0$ for each $0\leqslant \eta\leqslant \eta_0$.
\end{theorem}
\begin{proof}
From Theorem \ref{teoconv} and our hypotheses, it is clear that there exists $\delta>0$ such that
\begin{itemize}
\item[-] for each $x^\ast\in \mathcal{E}_0$ there exists a unique equilibrium point $x_\eta^\ast$ of $T_\eta$ in $B_{\delta/2}^X(x^\ast)$ ;
\item[-] $B_{\delta}^X(x)\cap B_{\delta}^X(y)=\varnothing$ for $x,y\in \mathcal{E}_0$ and $x\neq y$ and
\item[-]  $\norma{T_\eta-T_0}_{\mathcal{O}_\delta(\mathcal{E}_0)}\to 0$ as $\eta \to 0^+$.
\end{itemize}

Now assume that there exist a sequence $\eta_k\to 0^+$ as $k\to \infty$ and equilibrium points $x_k\in \mathcal{E}_{\eta_k}$ for each $k\in \mathbb{N}$ such that $x_k \notin \mathcal{O}_\delta(\mathcal{E}_0)$. From the precompactness of $\cup_{\eta\in [0,1]}\mathcal{E}_\eta$, we can extract a subsequence, which we call the same, and a point $x_0\in \overline{\cup_{\eta\in [0,1]}\mathcal{E}_\eta}$, such that $x_k\to x_0$ as $k\to \infty$. But hence
\begin{align*}
\|x_0-T_0(x_0)\|& \leqslant \|x_0-x_k\|+\|T_{\eta_k}(x_k)-T_{\eta_k}(x_0)\|+\|T_{\eta_k}(x_0)-T_0(x_0)\| \\
& \leqslant (1+\|T_{\eta_k}\|_{V,Lip})\|x_0-x_k\| + \|T_{\eta_k}-T_0\|_{V,\infty}\to 0, \hbox{ as } \eta\to 0^+,
\end{align*}
since $\norma{T_{\eta_k}-T_0}_V\to 0$ as $k\to \infty$, and hence $x_0$ an equilibrium point of $T_0$ that lies in $X\setminus \mathcal{O}_{\delta/2}(\mathcal{E}_0)$, which gives us a contradiction and completes the result.
\end{proof}

With this theorem, the next result is straightforward.

\begin{corollary} \label{cor:Continuity}
The family $\{\mathcal{E}_\eta\}_{0\leqslant \eta\leqslant 1}$ is continuous at $\eta=0$.
\end{corollary}

\section{Autonomous perturbations of invariant manifolds} \label{sec:APIM}

Our goal in this section is to prove the continuity at $\eta=0$  of the family of unstable sets $\{W^{{\rm u},\eta}_{loc}(x^{\ast}_\eta)\}_{\eta\in [0,1]}$ of maps $T_\eta$ near an $\mathfrak{L}$-hyperbolic point $x^\ast_\eta$, assuming that $T_\eta$ converges to $T_0$ as $\eta\to 0^+$ in the Lipschitz norm of \eqref{eq:NormaT}. Using the consideration of Section \ref{sec:InvManifolds}, it is sufficient to prove this continuity for {\it global invariant manifolds}, and that is the theory that will be present in this section.

\begin{remark}\label{obs}$ $
\begin{itemize}
  \item[\bf 1.] It is clear that if $T\in C(X)$ has an $\mathfrak{L}$-hyperbolic point $x^\ast$ and $g\colon X\to X$ is a bi-Lipschitz maps with $g(x^\ast)=x^\ast$, then $x^\ast$ is also an weakly hyperbolic point for $S=g^{-1}\circ T\circ g$ (see Proposition \ref{hu3}).
  Moreover, $W^{S,\rm i}(x^\ast)=g(W^{T,\rm i}(x^{\ast}))$ for $i=u,s$, where $W^{H,\rm i}(x^\ast)$ is the manifold associated to the map $H=S$ or $T$.

  \item[\bf 2.] Again using Proposition \ref{hu3}, the bi-Lipschitz map $g$ can be chosen such that $W^{S,\rm i}(x^\ast)=X_i$, for $i=u,s$.
\end{itemize}
\end{remark}

The following lemma is quite technical but very important. Hence, in order not to disrupt the line of study, we will leave its proof for the appendix, see \ref{Lemma52}.

\begin{lemma}\label{convvarlipfinal_lem}
Assume that $\{T_\eta\}_{\eta\in[0,1]}$ is a family of maps in $C(X)$ such that each $T_\eta$ has a global attractor $\mathcal{A}_\eta$ and $\cup_{\eta\in [0,1]}\mathcal{A}_\eta$ is bounded in $X$. Assume that there exists a neighborhood $U$ of $\cup_{\eta\in[0,1]}\mathcal{A}_\eta$ such that $\norma{T_\eta-T_0}_U\to 0$ as $\eta \to 0^+$. Also, assume that each $T_\eta$ has an $\mathfrak{L}$-hyperbolic point $x^\ast_\eta$ and the parameters given in Definition \ref{hyperbolic_fixed_point} can be taken uniformly with respect to $\eta\in[0,1]$\footnote[1]{This can be made true using Remark \ref{rem:UnifConstants}}. Suppose that $N_{0,u} = \pi_uN_0$ and $N_{0,s}=(I-\pi_u)N_0$ are such that for any fixed neighborhoods $W_u$ and $W_s$ of $0$ in $X_u$, $X_s$ respectively we have
\begin{equation} \label{hyp:N}
\|N_{0,u}\|_{\overline{B^{X_u}_r(0)}\times W_s,Lip} \to 0 \quad \hbox{ and } \|N_{0,s}\|_{W_u\times \overline{B^{X_s}_r(0)},Lip}\to 0 \quad \hbox{ as } r\to 0^+.
\end{equation}

Finally assume that $x^\ast_\eta\to x_0^\ast$ as $\eta\to 0^+$, then there exist a neighborhood $V_u$ of $0$ in $X_u$ and Lipschitz continuous maps $\theta_\eta\colon V_u\to X_s$ for sufficiently small $\eta$ such that
$$
W^{\rm u,\eta}_{loc}(x^{\ast}_\eta)=\{x_\eta^\ast+\xi+\theta_\eta(\xi)\colon \xi \in V_u\} \hbox{ and }\norma{\theta_\eta-\theta_0}_{V_u}\to 0, \hbox{ as } \eta\to 0^+.
$$

Analogously, there exist a neighborhood $V_s$ of $0$ in $X_s$ and maps $\sigma_\eta\colon X_s\to X_u$ such that
$$
W^{\rm s,\eta}_{loc}(x^{\ast}_\eta)=\{x_\eta^\ast+\sigma_\eta(\mu)+\mu\colon \mu \in V_s\}\hbox{ and }\norma{\sigma_\eta-\sigma_0}_{V_s}\to 0, \hbox{ as } \eta\to 0^+.
$$
\end{lemma}

\begin{remark}
The conditions presented in \eqref{hyp:N} hold, for instance when $T_0$ is a differentiable map in $X$ $($we say $T\in C^1(X))$ and ${\rm dim}X<\infty$ or in the case when $T_0$ is a differentiable in $X$  with uniformly continuous derivative in a neighborhood of the equilibrium point $($for the latter we say that $T\in C^{1+}(X))$.
\end{remark}

\begin{theorem}\label{convvarfinal*}
Let $\{T_\eta\}_{\eta\in [0,1]}$ be a family of maps in $C(X)$. Suppose that
\begin{itemize}
\item[\bf (a)] there exists $p\in \mathbb{N}$ such that each $T_\eta$ has a family $\mathcal{E}_\eta=\{x^\ast_{1,\eta},\ldots,x^{\ast}_{p,\eta}\}$ of $\mathfrak{L}$-hyperbolic equilibria;
\item[\bf (b)] $\|x^\ast_{i,\eta}-x^\ast_{i,0}\|\to 0$ as $\eta\to 0^+$, for $i=1,\ldots,p$;
\item[\bf (c)] there exist neighborhoods $U,V$ of $\ \mathcal{E}_0$ such that $\norma{T_\eta-T_0}_U\to 0$ as $\eta\to 0^+$ and $T_\eta \colon U\to T_\eta(U)$ is bi-Lipschitz for each $\eta\in [0,1]$.
\item[\bf (d)] $T_0\in C^{1+}(X)$.
\end{itemize}

Then for each $n\in \mathbb{N}$ and $i=1,\ldots,p$, the families $\{T^n_\eta W^{\rm u,\eta}_{loc}(x^\ast_{i,\eta})\}_{\eta\in[0,1]}$ and $\{W^{\rm s,\eta}_{loc}(x^\ast_{i,\eta})\}_{\eta\in [0,1]}$ are continuous at $\eta=0$, with the convergence in the norm $\norma{\cdot}_U$.
\end{theorem}
\begin{proof} The result follows from the continuity of the families $\{W^{\rm u,\eta}_{loc}(x^\ast_{i,\eta})\}_{\eta\in[0,1]}$ and  $\{W^{\rm s,\eta}_{loc}(x^\ast_{i,\eta})\}_{\eta\in[0,1]}$, given in Lemma \ref{convvarlipfinal_lem}, and the fact that the maps $T_\eta\colon U\to T_\eta(U)$ are bi-Lipschitz.
\end{proof}

\label{app:LipschitzMani}

One particular question that arises when we are dealing with invariant {\it Lipschitz manifolds}; that is, sets given locally as graphs of Lipschitz maps, is the following: if we make small autonomous perturbations of a Lipschitz manifold, can the perturbed manifold be locally represented by a graph of a Lipschitz map {\it with the same} domain as the limiting manifold? Under suitable conditions of convergence, the answer is {\it yes}, but the proof of this result is not trivial.

Since we will need this kind of result, which is not the main focus of this paper, we added them in Appendix \ref{app:LipschitzMani} for the sake of completeness. All the results used in the appendix we be referenced when used, so the reader can follow the theory without any loss if he/she chooses to skip the Appendix for now.

\section{Topological stability of discrete dynamically gradient maps} \label{sec:TS}

In this section, we discuss briefly the topological stability for discrete dynamically gradient maps, which is a known result in the literature (the reader may see \cite{aragao,aragao2,CL09} for detailed discussions on this subject), but is an important stepping stone to study the {\it geometrical stability}, which is the main goal of our work.

\begin{definition}
We say a set $\mathcal{A}\subset X$ is the \df{global attractor} for a map $T\in C(X)$ if $\mathcal{A}$ is the global attractor of the discrete semigroup $\{T^n\colon n\in \mathbb{N}\}$.
\end{definition}

\begin{definition}
Let $T\in C(X)$ with a global attractor $\mathcal{A}$. We say that a set $\Xi\subset \mathcal{A}$ is an \df{isolated invariant} for $T$ if $T(\Xi)=\Xi$ and there exists $r>0$ such that $\Xi$ is the maximal invariant set in $\mathcal{O}_r(\Xi)$; that is, if $B\subset \mathcal{O}_r(\Xi)$ satisfies $T(B)=B$ then $B\subset \Xi$.
\end{definition}

From the continuity of $T$ it is clear that $\overline{\Xi}$ is invariant for $T$, and hence the maximality of $\Xi$ in $\mathcal{O}_r(\Xi)$ implies that $\Xi$ is closed, and since $\Xi\subset \mathcal{A}$, $\Xi$ is compact.

\begin{definition}
Let $T\in C(X)$ with a global attractor $\mathcal{A}$. We say that a family $\mathfrak{E}=\{\Xi_1,\ldots,\Xi_p\}$ is a \df{disjoint family of isolated invariants} for $T$ if each $\Xi_i$ is an isolated invariant for $T$ and there exists $r_0>0$ such that $\mathcal{O}_{r_0}(\Xi_i)\cap \mathcal{O}_{r_0}(\Xi_j)=\varnothing$ for $1\leqslant i<j\leqslant p$.
\end{definition}

\begin{definition}
Let $T\in C(X)$ with a global attractor $\mathcal{A}$ and $\mathfrak{E}=\{\Xi_1,\ldots,\Xi_p\}$ a disjoint family of isolated invariants for $T$. A \df{heteroclinic structure} in $\mathfrak{E}$ is a subset $\{\Xi_{k_1},\ldots,\Xi_{k_\ell}\}$ of $\mathfrak{E}$ and bounded global solutions $\xi_i$ of $T$ for $i=1,\ldots,m$ such that
$$
\lim_{m\to -\infty}d(\xi_i(m),\Xi_{k_i}) = 0 \hbox{ and } \lim_{m\to \infty}d(\xi_i(m),\Xi_{k_{\ell+1}})=0, \hbox{ for each } i=1,\ldots,\ell,
$$
 where $\Xi_{k_{\ell+1}}$ is defined as $\Xi_{k_1}$.
\end{definition}

\begin{definition}
Let $T\in C(X)$ with a global attractor $\mathcal{A}$ and $\mathfrak{E}=\{\Xi_1,\ldots,\Xi_p\}$ a disjoint family of isolated invariants for $T$. We say that $T$ is \df{dynamically gradient} with respect to $\mathfrak{E}$ if it satisfies:
\begin{itemize}
\item[\bf (DG1)] given a bounded global solution $\xi$ of $T$, there exist $\Xi_i,\Xi_j\in \mathfrak{E}$ such that
$$
\lim_{m\to -\infty}d(\xi(m),\Xi_i) = 0 \hbox{ and } \lim_{m\to \infty}d(\xi(m),\Xi_j)=0;
$$
\item[\bf (DG2)] there are no heteroclinic structures in $\mathfrak{E}$.
\end{itemize}
\end{definition}

The goal here is to study the stability of this concept under small autonomous perturbations. To this end, we need the following definition:
\begin{definition}
Let $\{T_\eta\}_{\eta\in [0,1]}\subset C(X)$. We say that this family is \df{continuous} at $\eta=0$ if
$$
\max_{n=1,\ldots,N}\sup_{x\in K}\|T^n_\eta(x)-T^n_0(x)\| \to 0 \hbox{ as } \eta\to 0^+,
$$
for each compact subset $K$ of $X$ and $N\in \mathbb{N}$. We say that $\{T_\eta\}_{\eta\in [0,1]}$ is \df{collectively asymptotically compact} if given sequences $\eta_k\to 0^+$, $n_k\to \infty$ and $\{x_k\}$ bounded in $X$ such that $\{T^{n_k}_{\eta_k}(x_k)\}$ is bounded, then $\{T^{n_k}_{\eta_k}(x_k)\}$ has a convergent subsequence.
\end{definition}

With these definitions we are able to present the main result concerning the stability of the {\it dynamically gradient} concept.

\begin{proposition} \label{prop:DynGradient1}
Let $\{T_\eta\}_{\eta\in [0,1]}\in C(X)$ be a collectively asymptotically compact and continuous family of maps at $\eta=0$. Assume that:
\begin{itemize}
\item[\bf (a)] $T_\eta$ has a global attractor $\mathcal{A}_\eta$ for each $\eta\in [0,1]$ and $\cup_{\eta\in [0,1]}\mathcal{A}_\eta$ is precompact in $X$; \smallskip
\item[\bf (b)] there exists $p\in \mathbb{N}$ such that $T_\eta$ has a family of isolated invariants $\mathfrak{E}_\eta=\{\Xi_{1,\eta},\ldots,\Xi_{p,\eta}\}$ for each $\eta\in [0,1]$;\smallskip
\item[\bf (c)] $\displaystyle \max_{i=1,\ldots,p}\Big\{{\rm dist}_H(\Xi_{i,\eta},\Xi_{i,0})+{\rm dist}_H(\Xi_{i,0},\Xi_{i,\eta})\Big\}\to 0$ as $\eta\to 0^+$;\smallskip
\item[\bf (d)] there exists $\eta_0>0$ and neighborhoods $V_i$ of $\Xi_{i,0}$ such that $\Xi_{i,\eta}$ is the maximal invariant set for $T_\eta$ in $V_i$ for each $i=1,\ldots,p$ and $\eta\in [0,\eta_0]$
\item[\bf (e)] $T_0$ is dynamically gradient with respect to $\mathfrak{E}_0$.
\end{itemize}

Then there exists $\eta_1>0$ such that $T_\eta$ is dynamically gradient with respect to $\mathfrak{E}_\eta$ and
$$
\mathcal{A}_\eta = \bigcup_{i=1}^p W^{\rm u}(\Xi_{i,\eta}), \hbox{ for each }\eta\in [0,\eta_1],
$$
where
\begin{align*}
W^{\rm u}(\Xi_{i,\eta}) = \{x\in X\colon  \hbox{ there exists a global solution }& \xi \hbox{ of } T_\eta \hbox{ such that } \xi(0)=x \\ & \hbox{ and } d(\xi(m), \Xi_{i,\eta})\to 0 \hbox{ as } m\to -\infty\}.
\end{align*}
\end{proposition}
\begin{proof}
Apply \cite[Theorem 2.13]{CL09}.
\end{proof}

We can apply this result to our particular case to obtain the following:

\begin{proposition} \label{prop:TS}
Let $\{T_\eta\}_{\eta\in [0,1]}\in C(X)$ a collectively asymptotically compact and continuous family at $\eta=0$ and assume that:
\begin{itemize}
\item[\bf (a)] $T_\eta$ has a global attractor $\mathcal{A}_\eta$ for each $\eta\in [0,1]$ and $\cup_{\eta\in [0,1]}\mathcal{A}_\eta$ is precompact in $X$; \smallskip
\item[\bf (b)]  there is a neighborhood $U$ of $\cup_{\eta\in [0,1]}\mathcal{A}_\eta$ such that $\norma{T_\eta-T_0}_U\to 0$ as $\eta\to 0^+$;
\item[\bf (c)] there exists $p\in \mathbb{N}$ such that $T_\eta$ has a family of isolated invariants $\mathfrak{E}_\eta=\{x^\ast_{1,\eta},\ldots,x^\ast_{p,\eta}\}$ for each $\eta\in [0,1]$ consisting only of equilibria;\smallskip
\item[\bf (d)] all points in $\mathfrak{E}_0$ are $\mathfrak{L}$-hyperbolic.
\item[\bf (e)] $T_0$ is dynamically gradient with respect to $\mathfrak{E}_0$.
\end{itemize}

Then there exists $\eta_1>0$ such that $T_\eta$ is dynamically gradient with respect to $\mathfrak{E}_\eta$ and
$$
\mathcal{A}_\eta = \bigcup_{i=1}^p W^{\rm u}(x^\ast_{i,\eta}), \hbox{ for each }\eta\in [0,\eta_1],
$$
where
\begin{align*}
W^{\rm u}(x^\ast_{i,\eta}) = \{x\in X\colon  \hbox{ there exists a global solution }& \xi \hbox{ of } T_\eta \hbox{ such that } \xi(0)=x \\ & \hbox{ and } \|\xi(m)- x^\ast_{i,\eta}\|\to 0 \hbox{ as } m\to -\infty\}.
\end{align*}
\end{proposition}
\begin{proof} Using Corollary \ref {cor:Continuity} we see that all the hypotheses of Proposition \ref{prop:DynGradient1} are satisfied, and hence the result follows.
\end{proof}

\section{$\mathfrak{L}$-transversality} \label{sec:Trans}

When we are studying the geometrical stability of semigroups in the differentiable case (see \cite{Ang,BP,JG,hmo,Henry1,Raugel} for instance), two concepts are the key to unlock the most crucial results: hyperbolicity and transversality. The hyperbolicity in our case is translated to $\mathfrak{L}$-hyperbolicity, and we already proved that the main properties we obtain for hyperbolic points, we can also obtain for $\mathfrak{L}$-hyperbolic points. It is time now to extend the concept of {\it transversality} to $\mathfrak{L}$-transversality without assume differentiability property. This is our goal in this section, to define the notion of $\mathfrak{L}$-transversality and obtain properties of $\mathfrak{L}$-transversal manifolds, similar to the ones in the transversal case, which are necessary to obtain geometrical stability.

\begin{definition} \label{def:Ltransversal}
Let $X$ be a Banach space, $M,N\subset X$ and $x_0\in M\cap N$. We say that $M$ and $N$ are {\bf $\mathfrak{L}-$transversal at $x_0$} if there exist closed vector subspaces $X_1,X_2\subset X$, with $X=X_1\oplus X_2$, a real number $r>0$ and two Lipschitz continuous functions $\theta\colon B_r^{X_1}(0)\to X_2$ and $\sigma\colon B_r^{X_2}(0)\to X_1$, with $\theta(0)=\sigma(0)=0$, ${\rm Lip}(\theta)<1$, ${\rm Lip}(\sigma)<1$ and
$$
\{x_0+\xi+\theta(\xi)\colon \xi\in B_r^{X_1}(0)\}\subseteq M \quad \hbox{ and }\quad \{x_0+\sigma(\eta)+\eta\colon \eta \in B_r^{X_2}(0)\}\subseteq N.
$$

We denote it by $M\pitchfork_{\mathfrak{L},x_0} N$. If $M$ and $N$ are $\mathfrak{L}$-transversal for every $x_0\in M\cap N$, we say that $M$ and $N$ are $\mathfrak{L}$-transversal and denote it by $M\pitchfork_{\mathfrak{L}}N$.
\end{definition}

Note that if $M\cap N=\varnothing$ then $M$ and $N$ are $\mathfrak{L}$-transversal, by vacuity. Also, note that if $x^\ast$ is an weakly hyperbolic point of a map $T$ and, then $W^{\rm u}_{loc}(x^\ast)$ and $W^{\rm s}_{loc}(x^\ast)$ are $\mathfrak{L}$-transversal at $x^\ast$, and since $x^\ast$ is the only point in their intersection, they are $\mathfrak{L}$-transversal.

\begin{proposition}\label{lem6.3}
Let $X$ be a Banach space and $X_1,X_2$ closed subspaces of $X$ such that $X=X_1\oplus X_2$. Assume that there exist $r>0$ and functions $\theta,\tilde{\theta}\colon B_r^{X_1}(0)\to X_2$, $\sigma,\tilde{\sigma}\colon B_r^{X_2}(0)\to X_1$ with $\theta(0)=\sigma(0)=0$, ${\rm Lip}(\theta)<1$ and ${\rm Lip}(\sigma)<1$. Define the sets:
\begin{align*}
&M=\{y+\theta(y)\colon y\in {B_r^{X_1}(0)}\},\quad \quad\quad\quad N=\{\sigma(x)+x\colon x\in {B_r^{X_2}(0)}\}, \\
&
\tilde{M}=\{z+y+\tilde{\theta}(y)\colon y\in {B_r^{X_1}(0)}\} \hbox{ and } \ \tilde{N}=\{z+\tilde{\sigma}(x)+x\colon x\in {B_r^{X_2}(0)}\},
\end{align*}
where $z\in X$, and suppose also that there exists $0<c<1$ such that ${\rm Lip}(\theta)\leqslant c$, ${\rm Lip}(\sigma) \leqslant c$ and
\begin{equation} \label{eq:LipTodas}
  \|\theta(y)-\tilde{\theta}(y)\| \leqslant (1-c)\tfrac r2 \hbox{ and } \|\sigma(x)-\tilde{\sigma}(x)\|\leqslant (1-c)\tfrac r2 \hbox{ for all } y\in {B_r^{X_1}(0)}, \ x\in {B_r^{X_2}(0)},
\end{equation}
\begin{itemize}
  \item[\bf (a)] If ${\rm dim }X_1<\infty$ or ${\rm Lip}(\tilde{\theta})\cdot{\rm Lip}(\tilde{\sigma})<1$ then $\tilde{M}\cap \tilde{N}\neq \varnothing$.
  \item[\bf (b)] If ${\rm Lip}(\tilde{\theta})<1$ and ${\rm Lip}(\tilde{\sigma})<1$ then there exists a point $y_0$ such that $\tilde{M}\pitchfork_{\mathfrak{L},y_0} \tilde{N}$.
\end{itemize}
\end{proposition}
\begin{proof}
Let $K^1_r=\overline{B_{r/2}^{X_1}(0)}$ and $K^2_r=\overline{B_{r/2}^{X_2}(0)}$. Using \eqref{eq:LipTodas}, for $y\in K^1_r$ we have
$$
\|\tilde{\theta}(y)\|\leqslant \|\tilde{\theta}(y)-\theta(y)\|+\|\theta(y)\| \leqslant (1-c)\tfrac r2+c \tfrac r2  = \tfrac r2,
$$
and hence $\tilde{\theta}(K^1_r)\subset K^2_r$. Analogously, we obtain $\tilde{\sigma}(K^2_r)\subset K^1_r$.

 We see that $\tilde{M}$ and $\tilde{N}$ have non-empty intersection if there exist $y\in K^1_r$ and $x\in K^2_r$ such that $y+\tilde{\theta}(y)=\tilde{\sigma}(x)+x$. The latter is true if there exists $y\in K^1_r$ such that $\tilde{\sigma}(\tilde{\theta}(y))=y$; that is, the map $g\colon K^1_r\to K^1_r$ given by $g(y)=\tilde{\sigma}(\tilde{\theta}(y))$ has a fixed point.

Clearly the map $g$ is well-defined, for if $y\in K^1_r$ then
$$
\|g(y)\| = \|\tilde{\sigma}(\tilde{\theta}(y)) - \sigma(\tilde{\theta}(y))\|+\|\sigma(\tilde{\theta}(y))-\sigma(0)\| \leqslant (1-c)\tfrac r2+c\tfrac r2 = \tfrac r2.
$$

If ${\rm dim}X_1<\infty$ then Brouwer's Fixed Point Theorem implies that $g$ has a fixed point in $K^1_r$ and hence $\tilde{M}\cap \tilde{N}\neq \varnothing$.

Note that
$$
\|g(y_1)-g(y_2)\| = \|\tilde{\sigma}(\tilde{\theta}(y_1))-\tilde{\sigma}(\tilde{\theta}(y_1))\| \leqslant {\rm Lip}(\tilde{\sigma})\cdot {\rm Lip}(\tilde{\theta})\|y_1-y_2\|,
$$
for all $y_1,y_2\in K^1_r$. Therefore $g$ is a contraction and has a unique fixed point $y_1\in K^1_r$ and defining $y_2=\tilde{\theta}(y_1)$ we have $y_0=y_1+y_2\in \tilde{M}\cap \tilde{N}$. Note that, by construction, we also have $y_1=\tilde{\sigma}(y_2)$.

It remains to prove the $\mathfrak{L}$-transversality at $y_0$. To this end, firstly we choose $r_0>0$ such that $\overline{B_{r_0}^{X_1}(y_1)}\subset B_r^{X_1}(0)$ and $\overline{B_{r_0}^{X_2}(y_2)}\subset B_r^{X_2}(0)$. Now we define functions $\theta_\ast \colon B_{r_0}^{X_1}(0)\to X_2$ and $\sigma_\ast \colon B_{r_0}^{X_2}(0)\to X_1$ by $
\theta_\ast(y) = \tilde{\theta}(y+y_1) - y_2 = \tilde{\theta}(y+y_1) - \tilde{\theta}(y_1)$ and $\sigma_\ast(x) = \tilde{\sigma}(x+y_2)-y_1 = \tilde{\sigma}(x+y_2)-\tilde{\sigma}(y_2)$ for all $y\in B_{r_0}^{X_1}(0)$, $x\in  B_{r_0}^{X_2}(0)$. Therefore, the functions $\theta_\ast$ and $\sigma_\ast$ satisfy the conditions of Definition \ref{def:Ltransversal}, and hence $\tilde{M}\pitchfork_{\mathfrak{L},y_0} \tilde{N}$.
%
%
%
\end{proof}

\section{$\mathfrak{L}$-Morse Smale semigroups} \label{sec:LMSS}

In this section we develop the main concepts and results of our work, which involve the study of stability of certain structures, the {\it $\mathfrak{L}$-Morse Smale semigroups}, under small Lipschitz perturbations (in the norm $\norma{\cdot}$). We begin with a simple result, that will help us with the upcoming definitions.

\begin{proposition}
Let $T\in C(X)$ a map with a global attractor $\mathcal{A}$ and an $\mathfrak{L}$-hyperbolic point $x^\ast$ and local unstable manifold $W^{\rm u}_{loc}(x^\ast)$. If there exists neighborhoods $U,V$ of $\mathcal{A}$ in $X$ such that $T\colon U\to V$ is bi-Lipschitz, then the \df{unstable set} defined by
$$
W^{\rm u}(x^{\ast}) = \bigcup_{n\in\mathbb{N}} T^n(W^{\rm u}_{loc}(x^\ast)),
$$
is locally given as graphs of Lipschitz maps; in other words, it is a Lipschitz manifold. Moreover $x\in W^{\rm u}(x^{\ast})$ iff there exists a global bounded solution $\xi$ of $T$ with $\xi(0)=x$ and $\|\xi(m)-x^\ast\|\to 0$ as $m\to -\infty$.
\end{proposition}
\begin{proof}
If follows directly from the characterization of $W^{\rm u}_{loc}(x^\ast)$ in Corollary \ref{cor:US} and the bi-Lipschitz property of $T$.
\end{proof}

Now we can define the main concept of our paper.

\begin{definition}\label{WLSM}
Let $T\in \mathcal{C}(X)$ a map with a global attractor $\mathcal{A}$. We say that $T$ is \df{$\mathfrak{L}$-Morse-Smale}\index{l@$\mathfrak{L}$-Morse-Smale} $($or \df{$\mathfrak{L}$-MS}, for short$)$ if:
\begin{itemize}
  \item[\bf (i)] $T$ is dynamically gradient with respect to a finite family $\mathcal{E}=\{x^\ast_1,\ldots,x^{\ast}_p\}$ of $\mathfrak{L}$-hyperbolic points;
  \item[\bf (ii)] there exist neighborhoods $U$ of $\mathcal{E}$ in $X$ such that $T\colon U\to T(U)$ is bi-Lipschitz;
  \item[\bf (iii)] if $W^{\rm u}(x^\ast_{i})\cap W^{\rm s}_{loc}(x^*_{j})\neq \varnothing$ then there exists $n\in \N$ and $x_0\in X$ such that $T^nW_{loc}^{\rm u}(x^\ast_{i})\pitchfork_{\mathfrak{L},x_0} W^{\rm s}_{loc}(x^*_{j})$;
  \item[\bf (iv)] If $W^{\rm u}(x^*_{i})\cap W^{s}_{loc}(x^*_{j})\neq \varnothing$ and $W^{\rm u}(x^*_{j})\cap W^{\rm s}_{loc}(x^*_{k})\neq \varnothing$, then
  $W^{\rm u}(x^*_{i})\cap W^{\rm s}_{loc}(x^*_{k})\neq \varnothing$.
\end{itemize}
\end{definition}


\begin{remark}
Clearly if $T$ is a classical Morse-Smale map with only hyperbolic points as critical elements, it is a $\mathfrak{L}$-Morse Smale map. Condition {\rm (iv)} in this case is a simple application of the well known $\lambda$-lemma, which can be found in the manuscript of D. Henry $($the infinite dimensional case$)$ or in \cite{Palis} $($the finite dimensional case$)$.
\end{remark}

Note that if $T\in C(X)$ has a global attractor $\mathcal{A}$ and two $\mathfrak{L}$-hyperbolic points $x^\ast_1,x^\ast_2$ then $W^{\rm u}(x^*_{1})\cap W_{loc}^{\rm s}(x^*_{2})\neq \varnothing$ iff there exists a bounded global solution $\xi$ in $\mathcal{A}$ such that $\xi(t)\to x^\ast_1$ as $t\to -\infty$ and $\xi(t)\to x^\ast_2$ as $t\to \infty$; in other words, there exists a \df{connection} between $x^\ast_1$ and $x^\ast_2$.

\begin{definition}
Let $T_1,T_2\in \mathcal{C}(X)$ maps with global attractors $\mathcal{A}_1$ and $\mathcal{A}_2$, respectively. We say that $\mathcal{A}_1$ and $\mathcal{A}_2$ are \df{geometrically equivalent} if:
\begin{itemize}
  \item[\bf(i)] $T_i$ is a dynamically gradient semigroup with a family $\mathfrak{E}_i=\{x^{\ast,i}_1,\ldots,x^{\ast,i}_n\}$ of $\mathfrak{L}$-hyperbolic points for $i=1,2$.
  \item[\bf(ii)] there exists a bijection $\mathcal{B}\colon \mathfrak{E}_1\to \mathfrak{E}_2$ such that
  \begin{equation}\label{eq17}
W^{\rm u}(x^{*,1}_{i})\cap W_{loc}^{\rm s}(x^{*,1}_{j})\neq \varnothing \hbox{ iff }
W^{\rm u}(\mathcal{B}(x^{*,1}_{i}))\cap W^{\rm s}_{loc}(\mathcal{B}(x^{*,1}_{j}))\neq \varnothing \hbox{ for each } i,j=1,\ldots,n.
\end{equation}
\end{itemize}
\end{definition}

With the definition previous to this one, item (ii) can be rewritten by saying: \emph{there exists a connection between $x^{\ast,1}_i$ and $x^{\ast,2}_j$ if and only if there exists a connection between $\mathcal{B}(x^{*,1}_{i})$ and $\mathcal{B}(x^{*,2}_{j})$}. Also, we can reorder $\mathfrak{E}_2,$ if necessary, to assume that $\mathcal{B}(x^{*,1}_{i})=x^{\ast,2}_i$ for each $i=1,\ldots,n$.

An important result that will aid us in dealing with perturbations of $\mathfrak{L}$-{\rm MS} maps is the following:

\begin{lemma}\label{Laux}
Let $\{T_\eta\}_{\eta \in [0,1]}\subset C(X)$ a collectively asymptotically compact and continuous family at $\eta=0$. Let $\eta_k\to 0^+$, $a_k,b_k\to \infty$ in $\mathbb{N}$ and, for each $k\in \N$, $\xi_k\colon [-a_k,b_k]\to X$ be a solution of $T_{\eta_k}$. Assume that
$$
\bigcup_{k\in \N}\xi_{k}(t) \hbox{ is precompact in $X$ for each }t\in \mathbb{R}\quad \hbox{ and }\quad  \Xi=\bigcup_{k\in \mathbb{N}}\xi_k([-a_k,b_k]) \hbox{ is bounded in $X$}.
$$

Then there exist a subsequence $\{k_m\}$ of $\mathbb{N}$ and a global solution $\xi_0\colon \mathbb{R}\to X$ of $T_0$ such that $\xi_{k_m}(t) \to \xi_0(t)$ as $m\to \infty$, uniformly for $t$ in bounded subsets of $\mathbb{R}$, and $\xi_0(\mathbb{R})\subset \overline{\Xi}$.
\end{lemma}
\begin{proof} See \cite[Lemma 3.4]{CLRLibro}.
\end{proof}


\begin{proposition}\label{st4}
Let $\{T_\eta\}_{\eta \in [0,1]}\subset {C}(X)$ a collectively asymptotically compact and continuous family at $\eta=0$. Assume that
\begin{itemize}
\item[\bf (a)] $T_\eta$ has a global attractor $\mathcal{A}_\eta$ for each $\eta\in [0,1]$ and $\cup_{\eta\in [0,1]} \mathcal{A}_\eta$ is precompact in $X$;

\item[\bf (b)] there exists $p\in \mathbb{N}$ such that $T_\eta$ has a family of isolated invariants $\mathfrak{E}_\eta=\{x^\ast_{1,\eta},\ldots,x^\ast_{p,\eta}\}$ for each $\eta\in [0,1]$ consisting only of $\mathfrak{L}$-hyperbolic points and
$$
\max_{i=1,\ldots,p}\|x^{\ast}_{i,\eta}-x^\ast_{i,0}\| \to 0 \hbox{ as } \eta\to 0^+.
$$

\item[\bf (c)] $T_0$ is dynamically gradient with respect to $\mathfrak{E}_0$ and satisfies item (iv) of Definition \ref{WLSM}.
\end{itemize}

Then there exists $\eta_0>0$ such that if $W^{\rm u,\eta}(x^\ast_{i,\eta})\cap W^{\rm s,\eta}_{loc}(x^\ast_{j,\eta})\neq \varnothing$ for some $\eta\in [0,\eta_0]$ we have $W^{\rm u,0}(x^*_{i,0})\cap W^{\rm s,0}_{loc}(x^*_{j,0})\neq\varnothing$.
\end{proposition}
\begin{proof}
If the conclusion is false, there would be a sequence $\eta_k\to 0^+$ and points $x^{\ast}_{i,\eta_k},x^{\ast}_{j,\eta_k}\in \mathfrak{E}_{\eta_k}$ with $W^{\rm u,\eta_k}(x^*_{i,\eta_k})\cap W^{\rm s,\eta_k}_{loc}(x^*_{j,\eta_k})\neq \varnothing$ and $W^{\rm u,0}(x^*_{i,0})\cap W^{\rm s,0}_{loc}(x^*_{j,0})=\varnothing$. Hence for each $k\in \mathbb{N}$ there exists a global solution $\xi_k\colon \mathbb{R}\to X$ with
$$
\lim_{t\to -\infty}\xi_k(t) = x^{\ast}_{i,\eta_k} \hbox{ and } \lim_{t\to \infty}\xi_k(t) = x^{\ast}_{j,\eta_k}.
$$

Using Lemma \ref{Laux} we can extract a finite sequence $e_1,\ldots,e_\ell$ of points in $\mathfrak{E}_0$ with $e_1=x^{\ast}_{i,0}$ and $e_\ell=x^{\ast}_{j,0}$ and construct global solutions $\xi_{0,m}\colon \mathbb{R}\to X$ of $T_0$ such that
$$
\lim_{t\to -\infty}\xi_{0,m}(t) = e_m \quad  \hbox{ and } \quad \lim_{t\to \infty}\xi_{0,m}(t) = e_\ell\quad \hbox{ for each }m=1,\ldots,\ell-1,
$$

This implies that $W^{\rm u,0}(e_m)\cap W^{\rm s,0}_{loc}(e_{m+1})\neq \varnothing$ for each $m=1,\ldots,l$ and using item (iv) of Definition \ref{WLSM} iteratively we obtain $W^{\rm u,0}(e_1)\cap W^{\rm s,0}_{loc}(e_\ell)\neq \varnothing$, which gives us a contradiction and completes the proof.
\end{proof}


Until now, we have proved that if there is a sequence of connections between given equilibrium points in the perturbed problems then there will be a connection in the limit problem between the limit equilibrium points. Roughly speaking, it means that connections cannot vanish in the limit. But we also need the converse statement; that is, connections are maintained. If the limit problem has a connection, then the perturbed problems will also present one. To do this, we need the following two technical lemmas. The proof of the first one can be found in Section \ref{app:B1} of Appendix \ref{app:LipschitzMani} \footnote[1]{This proof requires other results presented and proved in Appendix \ref{app:LipschitzMani}.} and the second is analogous.

\begin{lemma}\label{prop6.9}
Let $X_u,X_s$ be closed subspaces of $X$ with $X=X_u\oplus X_s$, $\pi_u$ the canonical projection of $X$ into $X_u$. Let $V_u\subset X_u$ a neighborhood of $0$ in $X_u$ and maps $\psi_\eta\colon V_u\to X$ with $\norma{\psi_\eta-\psi_0}_{V_u}\leqslant c(\eta)$ for all $\eta \in [0,1]$, where $c(\eta)\to 0$ as $\eta\to 0^+$.

Also, assume that $\psi_0(V_u) = \{\xi+\theta_0(\xi)\colon \xi\in V_u\}$. Then, there exist $\eta_1>0$, a neighborhood $W_u$ of $0$ in $X_u$ and maps $\theta_\eta\colon W_u\to X$ for $\eta\in [0,\eta_1]$ such that $\psi_\eta(W_u)=\{\xi+\theta_\eta(\xi)\colon \xi \in W_u\}$
and
$$
{\rm Lip}(\theta_\eta)\leqslant \frac{{\rm Lip}(\theta_0)+c(\eta)}{1-c(\eta)} \hbox{ and } \|\theta_\eta-\theta_0\|_{V_u\cap W_u,\infty}\leqslant (1+{\rm Lip}(\theta_0))c(\eta).
$$
\end{lemma}

\begin{lemma}\label{prop6.10}
Let $X_u,X_s$ be closed subspaces of $X$ with $X=X_u\oplus X_s$, $\pi_s$ the canonical projection of $X$ into $X_s$. Let $V_s\subset X_s$ a neighborhood of $0$ in $X_s$ and maps $\varphi_\eta\colon V_s\to X$ with $\norma{\varphi_\eta-\varphi_0}_{V_s}\leqslant c(\eta)$ for all $\eta \in [0,1]$, where $c(\eta)\to 0$ as $\eta\to 0^+$.

Also, assume that $\varphi_0(V_s) = \{\sigma_0(\mu)+\mu\colon \mu \in V_s\}$. Then there exist $\eta_1>0$, a neighborhood $W_s$ of $0$ in $X_s$ and maps $\sigma_\eta\colon W_s\to X$ for $\eta\in [0,\eta_1]$ such that $\varphi_\eta(W_s)=\{\sigma_\eta(\mu)+\mu \colon \mu \in W_s\}$
and
$$
{\rm Lip}(\sigma_\eta)\leqslant \frac{{\rm Lip}(\sigma_0)+c(\eta)}{1-c(\eta)} \hbox{ and } \|\sigma_\eta-\sigma_0\|_{V_s\cap W_s,\infty}\leqslant (1+{\rm Lip}(\sigma_0))c(\eta).
$$
\end{lemma}

With these two results and the Proposition \ref{st4} we can prove the main theorem of our work.

\begin{theorem}\label{st3}
Let $\{T_\eta\}_{\eta\in [0,1]}$ be a collectively asymptotically compact and continuous family of maps at $\eta=0$ in $C(X)$. Suppose that
\begin{itemize}
\item[\bf (a)] $T_\eta$ has a global attractor $\mathcal{A}_\eta$ for each $\eta\in[0,1]$ and $\cup_{\eta\in [0,1]}\mathcal{A}_\eta$ is precompact in $X$;
\item[\bf (b)] there exists $p\in \mathbb{N}$ such that $T_\eta$ has a family of isolated invariants $\mathcal{E}_\eta=\{x^\ast_{1,\eta},\ldots,x^\ast_{p,\eta}\}$ for each $\eta\in[0,1]$ consisting only of $\mathfrak{L}$-hyperbolic points and
$$
\max_{i=1,\ldots,p}\|x^\ast_{i,\eta}-x^\ast_{i,0}\|\to 0 \hbox{ as } \eta\to 0^+ \hbox{ for }i=1,\ldots,p;
$$
\item[\bf (c)] there exist neighborhood $U$ of $\cup_{\eta\in [0,1]}\mathcal{E}_\eta$
such that $\norma{T_\eta-T_0}_U\to 0$ as $\eta\to 0^+$ and
$T_\eta \colon U\to T_\eta(U)$ is bi-Lipschitz for each $\eta\in [0,1]$.
\item[\bf (d)] {\color{black}$T_0$ is a $\mathfrak{L}$-Morse-Smale map and $T_0\in C^{1+}(\mathcal{O}_\delta(\mathcal{E}_0),X)$ for some $\delta>0$.}
\end{itemize}


Then there exists $\eta_0>0$ such that $T_\eta$ is a $\mathfrak{L}$-Morse-Smale map with $\mathcal{A}_\eta$ geometrically equivalent to $\mathcal{A}_0$ for all $\eta\in [0,\eta_0]$.

%
\end{theorem}
\begin{proof} Since $T_0$ is $\mathfrak{L}$-Morse-Smale, there exists a point $x_0\in W^{\rm u,0}(x^*_{i,0})\cap W^{\rm s,0}_{loc}(x^*_{j,0})$, and the intersection at this point is $\mathfrak{L}$-transversal. Hence there exists a decomposition $X=X_1\oplus X_2$, $r>0$ and maps $\theta_0\colon B_r^{X_1}(0)\to X_2$, $\sigma_0\colon B_r^{X_2}(0)\to X_1$ with $\theta_0(0)=\sigma_0(0)=0$, ${\rm Lip}(\theta_0)<1$ and ${\rm Lip}(\sigma_0)<1$ such that
$$
\{x_0+\xi+\theta_0(\xi)\colon \xi\in B_r^{X_1}(0)\}\subset T_0^nW_{loc}^{\rm u,0}(x^*_{i,0}) \hbox{ and } \{x_0+\sigma_0(\mu)+\mu \colon \mu\in B_r^{X_2}(0)\}\subset W^{\rm s,0}_{loc}(x^*_{j,0}).
$$

By making a translation, we may assume that $x_0=0$. Using Theorem \ref{convvarfinal*} and Lemmas \ref{prop6.9}  and \ref{prop6.10}, there exist $\eta_1>0$ and $0<r_0<r$, maps $\theta_\eta\colon B_{r_0}^{X_1}(0)\to X_2$, $\sigma_\eta\colon B_{r_0}^{X_2}(0)\to X_1$ for $\eta\in [0,\eta_0]$ with
${\rm Lip}(\theta_\eta)\to {\rm Lip}(\theta_0)$, ${\rm Lip}(\sigma_\eta)\to {\rm Lip}(\sigma_0)$, $ \|\theta_\eta-\theta_0\|_{B_{r_0}^{X_1}(0),\infty}\to 0$ and $\|\sigma_\eta-\sigma_0\|_{B_{r_0}^{X_2}(0),\infty}\to 0$ as $\eta\to 0^+$ such that
$$
\{x_0 + \xi +\theta_\eta(\xi)\colon \xi \in B_{r_0}^{X_1}(0)\}\subset T_\eta^nW_{loc}^{\rm u,\eta}(x^*_{i,\eta}) \hbox{ and } \{x_0 + \sigma_\eta(\mu)+\mu\colon \mu \in B_{r_0}^{X_2}(0)\}\subset W^{\rm s,\eta}_{loc}(x^*_{i,\eta}),
$$
for each $\eta\in[0,\eta_0]$.
Therefore, from Proposition \ref{lem6.3} item {\rm (b)}, there exists $x_\eta$ such that $W^{\rm u,\eta}(x^*_{i,\eta})\pitchfork_{\mathfrak{L},x_\eta} W^{\rm s,\eta}_{loc}(x^*_{j,\eta})$ for each $\eta$ sufficiently small.
\end{proof}

Then there exists $\eta_0>0$ such that $T_\eta$ is a $\mathfrak{L}$-Morse-Smale map with $\mathcal{A}_\eta$ geometrically equivalent to $\mathcal{A}_0$ for all $\eta\in [0,\eta_0]$.

\section{Example} \label{sec:Ex}
Consider the following family of autonomous partial differential equations given by
\begin{equation}\label{E1}
\left\{
\begin{array}{l}
u_t- u_{xx} = \lambda( u-u^3)+\eta \sin (u_x), \quad x\in (0,\pi), \ t>0\\
u(t,0)=u(t,\pi)=0, \quad t>0\\
u(0,x)=u_0(x), \quad x\in[0,\pi],
\end{array}
\right.
\end{equation}
where $\eta\in [0,1]$ and $\lambda>0$. Let $X=L^2(0,\pi)$ with norm $\|\cdot\|$, $-\Delta=A\colon D(A)\subset X\to X$ is the negative Dirichlet Laplacian with $D(A)=H^1_0(0,\pi)\cap H^2(0,\pi)$, and $X^{\alpha/2}$ the fractional power space of $X$ with norm $\|\cdot\|_\alpha:=\|A^{\alpha/2}(\cdot)\|$, for $\alpha\in \mathbb{R}$. We can write the problem as an abstract evolution equation, given by
\begin{equation}\label{E2}
\left\{
\begin{array}{l}
u_t+A u = f(u) +F_\eta(u), \ t>0\\
u(0)=u_0 \in X^{1/2},
\end{array}
\right.
\end{equation}
where $f(u)(x) = \lambda u(x) - u^3(x)$ and $F_\eta(u)(x)=\eta \sin(u_x(x))$ for each $x\in [0,\pi]$. Clearly for each $\eta$, the map $F_\eta$ defines a bounded and globally Lipschitz operator from $X^{1/2}$ to $X$, since
\begin{equation} \label{F1}
\|F_\eta(u)\|^2=\int_\Omega |F_\eta (u)(x)|^2dx \leqslant \eta^2 |\Omega|,
\end{equation}
and
\begin{equation} \label{F2}
\|F_\eta (u) - F_\eta (v)\|^2=\int_\Omega|F_\eta (u)(x)-F_\eta(v)(x)|^2dx\leqslant \eta^2 \|u-v\|_1^2,
\end{equation}
for all $u,v\in X$.

For $\eta =0$ we have the Chafee-Infante equation (see \cite{CI}), which is well-posed in $X^{1/2}$ and the solutions exist for all positive time, and as $t\to \infty$ each solution $u(t,\cdot)$ of \eqref{E2} with $\eta=0$ converges in $X^{1/2}$ to an equilibrium solution $\phi$ which satisfies
\begin{equation*}
\left\{
\begin{array}{l}
\phi ''(x) +\lambda( \phi(x) -\phi^3(x) )=0, \quad x\in (0,\pi)\\
\phi(0)=\phi(\pi)=0.
\end{array}
\right.
\end{equation*}

Also they prove that there are only a finite number of such equilibria. In fact if $n^2< \lambda \leqslant (n+1)^2$ there are exactly $2n+1$ equilibria, where $n$ is a nonnegative integer. Moreover, if $0<\lambda \leqslant 1$, the only equilibrium is the zero solution which is globally asymptotically stable. For $\lambda>1$ the zero solution is unstable and also all the others except for two, denotes by $\phi_1^+$ and $\phi_1^-$. These two solutions are characterized by the fact that $\phi_1^-(x)<0<\phi_1^+(x)$ for all $x\in (0,\pi)$ and these solutions are asymptotically stable.

Using \eqref{F1} and \eqref{F2} and the results of \cite{henry} we know that problem \eqref{E2} is also well-posed in $X^{1/2}$ and the solutions exist for all positive times. Hence for each $\eta\in[0,1]$ we obtain a semigroup $\{T_\eta(t)\colon t\geqslant 0\}$ in $X^{1/2}$ and
\begin{equation}\label{eq:Semi}
\begin{aligned}
T_\eta(t)u_0 = e^{-At}u_0 +\int_0^t e^{-A(t-s)}f(T_\eta(s)u_0)ds + \int_0^t e^{-A(t-s)}F_\eta(T_\eta(s)u_0)ds,
\end{aligned}
\end{equation}
for all $t\geqslant 0$ and $u_0\in X^{1/2}$.

The semigroup $\{T_0(t)\colon t\geqslant 0\}$ is the solution of \eqref{E2} and is given by
$$
T_0(t)u_0 = e^{-At}u_0 +\int_0^t e^{-A(t-s)}f(T_0(s)u_0)ds, \hbox{ for all }t\geqslant 0 \hbox{ and } u_0\in X^{1/2}.
$$

Using \cite{Henry1}, we know that for each $\lambda \notin \{1^2,2^2,3^2,\ldots\}$ the time one map $T_0=T_0(1)$ is a $C^2$ Morse-Smale map.
{\color{black}
Let $B=B_r^{X^{1/2}}(0)$ with $r>0$ such that $\mathcal{A}_0\subset \subset B$.
Let $g:\R^+\to [0,1]$, $g\in C^{\infty}(\R^+)$, such that $g([0,r])=\{1\}$ and $g([r+1,\infty))=\{0\}$.
Now we will denote
\begin{equation}\label{eq:semigrupo1}
T_0(t)u_0 = e^{-At}u_0 +\int_0^t e^{-A(t-s)}\tilde{f}(T_0(s)u_0)ds, \hbox{ for all } u_0\in X^{1/2}, t\geqslant 0.
\end{equation}
where $\tilde{f}(x):=g(\|x\|_{1})f(x)$ for $x\in X$. We note that the class of
differenciability of the $\tilde{f}:X^{1/2}\to X$ is the same of the $f:X^{1/2}\to X$
since $g$ is $C^{\infty}$, $X^{1/2}\backslash \{0\}\ni\mapsto \|x\|_{1}\in\R^{+}$ is $C^{\infty}$
(because $X^{1/2}$ is a Hilbert space) and $g(\|\cdot\|_1):X^{1/2}\to \R^+$ is constant
in a neighborhood of the point that $\|\cdot\|_1$ loses differenciability.

Now we denote $T_0:=T_0(1)$. Thus, we have that $\tilde{f}$ is
bounded and globally Lipschitz.
Note that, $T_0$ continues  automatically a
Morse-Smale semigroup which is $C^2$ on $X$.

We denote
\begin{equation}\label{eq:Semi1}
\begin{aligned}
T_\eta(t)u_0 = e^{-At}u_0 +\int_0^t e^{-A(t-s)}\tilde{f}(T_\eta(s)u_0)ds + \int_0^t e^{-A(t-s)}F_\eta(T_\eta(s)u_0)ds,
\end{aligned}
\end{equation}
and $T_\eta:=T_\eta(1)$.

}

Now, using the results in Section \ref{app:Diff} of the Appendix, we are able to prove the following.

\begin{proposition}
The function $F_\eta\colon X^{1/2}\to X$ is not Fr\'echet-differentiable at any point of $X^{1/2}$, for $\eta\in (0,1]$.
\end{proposition}
\begin{proof}
For $\eta\in (0,1]$ define the Nemytskii operator $G_\eta\colon X\to X$ by $G_\eta(u)(x)=\eta \sin (u(x))$, and clearly $F_\eta (u) = G_\eta(u_x)$ for each $u\in X^{1/2}$. Since $\partial_x\colon X^{1/2}\to X$ is an isometry and $F_\eta = G_\eta \circ \partial_x$ we have that if $F_\eta$ differentiable at $u_0\in X^{1/2}$ then $G_\eta$ is differentiable at $\partial_xu_0 = (u_0)_x \in X$. This contradicts Theorem \ref{theo:Diff}, since $G_\eta$ does not arise from an affine function.
\end{proof}

Using this proposition, one can see that the theory of small autonomous perturbations of semigroups cannot be applied to obtain geometrical stability of the family of semigroups $\{T_\eta(t)\colon t\geqslant 0\}$, since the perturbation is not continuously differentiable and hence the semigroups $\{T_\eta(t)\colon t\geqslant 0\}$ are not differentiable for $\eta\in(0,1]$. However, we are able to use our results to give a geometrical characterization of the global attractors of the perturbed semigroups.

Since we have a bounded and globally Lipschitz continuous perturbation, with Lipschitz constant less than or equal $\eta$, we can easily obtain that the family of time one maps $\{T_\eta(1)\}_{\eta \in [0,1]}$ is collectively asymptotically compact and continuous at $\eta=0$ in $C(H^1(0,\pi))$. Moreover, we have a global attractor $\mathcal{A}_\eta$ for each $\eta \in [0,1]$ (or sufficiently small, if necessary) such that $\cup_{\eta \in [0,1]}\mathcal{A}_\eta$ precompact in $H^1(0,\pi)$.

{\color{black}
Its easy to see that the Lipschitz convergence of the semigroups on bounded sets, which implies the item {\bf(c)} of the
Theorem \ref{st3}, follows from variational of constants formula and the fact that the nonlinearities
are globally Lipschitz and globally bounded.

Since $\cup_{\eta \in [0,1]}\mathcal{A}_\eta$ is precompact in $H^1(0,\pi)$
and $\|T_\eta-T_0\|_{U,\infty}\to 0$ as $\eta\to 0^+$ for every bounded set in $X^{1/2}$, we have that
$\{\mathcal{E}_{\eta}\}_{\eta\in[0,1]}$ is upper semicontinuous at $\eta=0$.
On the other hand, we have $\|T_\eta-T_0\|_{U,Lip}\stackrel{\eta\to 0}{\longrightarrow}0$
for $U$ bounded in $X^{1/2}$. Thus, from
Corollary \ref{convvarlip}, there exists $p\in \N$ and $\eta_0>0$ such that
$\mathcal{E}_\eta$ contains only $\mathfrak{L}$-hyperbolic fixed points and
$\mathcal{E}_\eta=\{x_{1,\eta}^*,\hdots,x_{p,\eta}^*\}$ which satisfies
\[
\max_{i=1,\hdots,p}\|x_{i,\eta}^*-x_{i,0}^*\|\stackrel{\eta\to 0}{\longrightarrow}0,
\]
i e, the item {\bf(b)} of the Theorem \ref{st3} is hold. Now we can aply the Theorem \ref{st3}
in order to conclude your example.

}

\begin{appendix}
\section{Technical results} \label{app:A}

\subsection{Proof of Proposition \ref{hu3}} \label{app:ProofHu3} $ $

{\sc $\ast$ Case 1:} Suppose that $\delta=\infty$.

{\bf \underline{Step 1.}}  Define $\tilde{T}=h^{-1}\circ T\circ h\colon X_u\times X_s\to X_u\times X_s$ with $h:X_u\times X_s\to X$ by $h(\xi,\eta)=\xi+\sigma(\eta)+\eta$. So $h^{-1}$ is well defined since
\[
\|h(\xi,\eta)-h(\la,\mu)\|\geqslant \|(\xi+\eta)-(\la+\mu)\|-\tn{Lip}(\sigma)\|\eta-\mu\|\geqslant (1-\tn{Lip}(\sigma))\|(\xi+\eta)-(\la+\mu)\|
\]
and $\tn{Lip}(\sigma)<1$. Let us show that $\tilde{T}(x)=(L_u\xi+\tilde{N}_u(x),L_s\eta+\tilde{N}_s(x))$ with $x=\xi+\eta$ and
\begin{equation*}
\left\{
\begin{array}{l}
\tilde{N}_u(\xi,\eta)=N_u(\xi+\sigma(\eta)+\eta)-N_u(\sigma(\eta)+\eta)+\sigma(\tilde{\eta})-\sigma(\hat{\eta});\\
\tilde{N}_s(\xi,\eta)=N_s(\xi+\sigma(\eta)+\eta).
\end{array}
\right.
\end{equation*}
such that $\tilde{N}_u(\eta)=0$ for $\eta\in X_{s}$. We know that there exists Lipschitz maps $\theta:X_{u}\rightarrow X_{s}$ and  $\sigma:X_{s}\rightarrow X_{u}$, such that $W^{\rm u}=\{\xi+\theta(\xi): \ \xi \in X_{u}\} \hbox{ and } W^{\rm s}=\{\sigma(\eta)+\eta: \ \eta \in X_{s}\}$, by Proposition \ref{teoUS}. Then if $T(h(\xi+\eta))=h(\hat{\xi}+\hat{\eta})$ we have
$$
T(\xi+\sigma(\eta)+\eta)=L_u\xi+L_u\sigma(\eta)+L_s\eta+N_u(\xi+\sigma(\eta)+\eta)+N_s(\xi+\sigma(\eta)+\eta)=\hat{\xi}+\sigma(\hat{\eta})+\hat{\eta},
$$
and thus
$$
\begin{cases}
\hat{\xi}=L_u\xi+L_u\sigma(\eta)+N_u(\xi+\sigma(\eta)+\eta)-\sigma(\hat{\eta})\\
\hat{\eta}=L_s\eta+N_s(\xi+\sigma(\eta)+\eta)
\end{cases}
$$

For $\sigma:X_{s}\rightarrow X_{u}$ we have
$$
\begin{cases}
\sigma(\tilde{\eta})=L_u\sigma(\eta)+N_u(\sigma(\eta)+\eta)\\
\tilde{\eta}=L_s\eta+N_s(\sigma(\eta)+\eta)
\end{cases}
$$
and hence we can rewrite the previous equation as
$$
\begin{cases}
\hat{\xi}=L_u\xi+N_u(\xi+\sigma(\eta)+\eta)-N_u(\sigma(\eta)+\eta)+\sigma(\tilde{\eta})-\sigma(\hat{\eta})\\
\hat{\eta}=L_s\eta+N_s(\xi+\sigma(\eta)+\eta),
\end{cases}
$$

Thus defining $\tilde{N}_u(\xi,\eta)=N_u(\xi+\sigma(\eta)+\eta)-N_u(\sigma(\eta)+\eta)+\sigma(\tilde{\eta})-\sigma(\hat{\eta})$ and $\tilde{N}_s(\xi,\eta)=N_s(\xi+\sigma(\eta)+\eta)$, we obtain the result.

\

Let us show that $\tn{Lip}(\tilde{N}_u),\tn{Lip}(\tilde{N}_s)\leqslant f(\gamma)$ with $f(\gamma)\stackrel{\gamma\to 0}{\to 0}$. We have
\begin{equation}\label{eqhu1}
\begin{array}{rl}
\|\tilde{N}_u(\xi,\eta)-\tilde{N}_u(\la,\mu)\|\leqslant & \|N_u(\xi+\sigma(\eta)+\eta)-N_u(\la+\sigma(\mu)+\mu)\|\\
                                                 + & \|N_u(\sigma(\eta)+\eta)-N_u(\sigma(\mu)+\mu)\|+\|\sigma(\tilde{\eta})-\sigma(\tilde{\mu})\|\\
                                                 + & \|\sigma(\hat{\eta})-\sigma(\hat{\mu})\|,
\end{array}
\end{equation}
but
\begin{equation*}
\begin{array}{rl}
\|N_u(\xi+\sigma(\eta)+\eta)-N_u(\la+\sigma(\mu)+\mu)\|\leqslant & \gamma[\|(\xi+\eta)-(\la+\mu)\|+\tn{Lip}(\sigma)\|\eta-\mu\|]\\
\leqslant & \gamma(1+\tn{Lip}(\sigma))\|(\xi,\eta)-(\la,\mu)\|,
\end{array}
\end{equation*}
and in particular
\[
\|N_u(\sigma(\eta)+\eta)-N_u(\sigma(\mu)+\mu)\|\leqslant \gamma(1+\tn{Lip}(\sigma))\|(\xi,\eta)-(\la,\mu)\|.
\]

Now
\begin{equation*}
\begin{array}{rl}
\|\sigma(\tilde{\eta})-\sigma(\tilde{\mu})\|\leqslant & \tn{Lip}(\sigma)\|\tilde{\eta}-\tilde{\mu}\|\\
= & \tn{Lip}(\sigma)\|(L_s\eta+N_s(\sigma(\eta)+\eta))-(L_s\mu+N_s(\sigma(\mu)+\mu))\|\\
\leqslant & \tn{Lip}(\sigma)[b\|\eta-\mu\|+\gamma(1+\tn{Lip}(\sigma))\|\eta-\mu\|]\\
\leqslant & (b+2\gamma)\tn{Lip}(\sigma)\|(\xi,\eta)-(\la,\mu)\|.
\end{array}
\end{equation*}
and
\begin{equation*}
\begin{array}{rl}
\|\sigma(\hat{\eta})-\sigma(\hat{\mu})\|\leqslant & \tn{Lip}(\sigma)\|\hat{\eta}-\hat{\mu}\|\\
= & \|(L_s\eta+N_s(\xi+\sigma(\eta)+\eta))-(L_s\mu+N_s(\la+\sigma(\mu)+\mu))\|\\
\leqslant & \tn{Lip}(\sigma)\left[b\|\eta-\mu\|+\gamma(\|(\xi+\eta)-(\la+\mu)\|+\tn{Lip}(\sigma)\|\eta-\mu\|)\right]\\
\leqslant & (b+2\gamma)\tn{Lip}(\sigma)\|(\xi+\eta)-(\la+\mu)\|.
\end{array}
\end{equation*}

Therefore $\tn{Lip}(\tilde{N_u})\leqslant 2[\gamma(1+\tn{Lip}(\sigma))+(b+2\gamma)\tn{Lip}(\sigma)]$, and
since $\tn{Lip}(\sigma)\leqslant \tfrac{\gamma}{a-b-3\gamma}$, we have $\tn{Lip}(\tilde{N_u})\leqslant \tfrac{2a\gamma}{a-b-3\gamma}=f(\gamma)$. Analogously $\tn{Lip}(\tilde{N_s})\leqslant \gamma(1+\tn{Lip}(\sigma))\leqslant f(\gamma)$.  Since $f(\gamma){\to 0}$ as $\gamma\to 0^+$, it follows that there exist $\gamma_0=\gamma_0(L,a,b)$ such that
\begin{equation}\label{eqhu2}
\|(I-L)^{-1}\|\cdot f(\gamma)<1 \quad   \hbox{ and } \quad  b+2f(\gamma)<1<a-2f(\gamma),
\end{equation}
 for all  $\gamma\in (0,\gamma_0]$. Since $(I-L)^{-1}x=(I-L_u)^{-1}x_{u}+(I-L_s)^{-1}x_{s}$, we have $\|(I-L)^{-1}\|\leqslant \tfrac{1}{1-b}+\tfrac{a}{a-1}$.
Thus, $\gamma_0=\gamma_0(a,b)$ and then $0$ is a weakly $\mathfrak{L}$-hyperbolic equilibrium for $\tilde{T}$.
\smallskip

{\bf \underline{Step 2.}} Let $\gamma\in (0,\gamma_0]$ as in \eqref{eqhu2}. Define $S=k^{-1}\circ \tilde{T}\circ k=g^{-1}\circ T\circ g$ with $k\colon X\to X_{u}\times X_{s}$ by $k(\xi+\eta)=(\xi,\tilde{\theta}(\xi)+\eta)$, $g=h\circ k$ with $\tilde{\theta}\colon X_{u}\to X_{s}$, $\tn{Lip}(\tilde{\theta})\leqslant \dfrac{f(\gamma)}{a-b-3f(\gamma)}= f_*(\gamma)<1$ and $W^{u}(\tilde{T},0)=\{\xi+\tilde{\theta}(\xi):\xi\in X_{u}\}$.
We will show that here exist $\gamma^*=\gamma^*(a,b)>0$ such that if $\gamma<\gamma^*$, then $g$ is bi-Lipschitz and $S$ is well defined.

Note that $g(\xi+\eta)=\xi+\eta+\sigma(\tilde{\theta}(\xi)+\eta)+\tilde{\theta}(\xi)$ and
\begin{equation*}
\begin{array}{rl}
\|g(\xi+\eta)-g(\la+\mu)\|\geqslant &\!\!\! \|(\xi+\eta)-(\la+\mu)\|- \tn{Lip}(\sigma)[\tn{Lip}(\tilde{\theta})\|\xi-\la\|+\|\eta-\mu\|]-\tn{Lip}(\tilde{\theta})\|\xi-\la\|\\
                          \geqslant &\!\!\! [1-2\tn{Lip}(\tilde{\theta})(1+\tn{Lip}(\sigma))]\|(\xi+\eta)-(\la+\mu)\|.
\end{array}
\end{equation*}

Thus, if $\gamma^*>0$ is such that $2f_*(\gamma)(\frac{a-b-2\gamma}{a-b-3\gamma})<1$ for all $\gamma\in (0,\gamma^*)$, we have $g$ bi-lipschitz and $\tn{Lip}(g^{-1})\leqslant \left[1-2f_*(\gamma)(\frac{a-b-2\gamma}{a-b-3\gamma})\right]^{-1}$. Therefore $\|g^{-1}(x)\|\leqslant \delta_2\|x\|$ with $\delta_2=\left[1-2f_*(\gamma)(\frac{a-b-2\gamma}{a-b-3\gamma})\right]^{-1}$.

Let us show that $S=L+\hat{N}$ with $\tn{Lip}(\hat{N})\leqslant f_1(\gamma)$ and $f_1(\gamma)\stackrel{\gamma\to 0}{\to 0}$. If $x=\xi+\eta, \hat{x}=\hat{\xi}+\hat{\eta}$ with $\xi,\hat{\xi}\in X_{u}$ and $\eta,\hat{\eta}\in X_{s}$ then $Sx=\hat{x}$ iff $\tilde{T}k(x)=k(\hat{x})$, which is true iff
$$
\tilde{T}(\xi,\tilde{\theta}(\xi)+\eta)=(\hat{\xi},\tilde{\theta}(\hat{\xi})+\hat{\eta}).
$$

Thus
\begin{equation*}
\left\{
\begin{array}{l}
\hat{\xi}=L_u\xi+\tilde{N_u}(\xi,\tilde{\theta}(\xi)+\eta);\\
\tilde{\theta}(\hat{\xi})+\hat{\eta}=L_s(\tilde{\theta}(\xi)+\eta)+\tilde{N_s}(\xi,\tilde{\theta}(\xi)+\eta).
\end{array}
\right.
\end{equation*}

On the other hand, $\tilde{\theta}$ satisfies
\begin{equation*}
\left\{
\begin{array}{l}
\tilde{\xi}=L_u\xi+\tilde{N_u}(\xi,\tilde{\theta}(\xi));\\
\tilde{\theta}(\tilde{\xi})=L_s\tilde{\theta}(\xi)+\tilde{N_s}(\xi,\tilde{\theta}(\xi))
\end{array}
\right.
\end{equation*}
and then$ \hat{\eta}=L_s\eta+\tilde{N_s}(\xi,\tilde{\theta}(\xi)+\eta)-\tilde{N_s}(\xi,\tilde{\theta}(\xi))+\tilde{\theta}(\tilde{\xi})-\tilde{\theta}(\hat{\xi})$; that is, $S=L+\hat{N}$ with $\hat{N}=\hat{N_u}+\hat{N_s}$, $\hat{N_u}(x)=\tilde{N_u}(\xi,\tilde{\theta}(\xi)+\eta)$ and $\hat{N_s}(x)=\tilde{N_s}(\xi,\tilde{\theta}(\xi)+\eta)-\tilde{N_s}(\xi,\tilde{\theta}(\xi))+\tilde{\theta}(\tilde{\xi})-\tilde{\theta}(\hat{\xi})$.  Note that $\hat{N_u}(\eta)=\tilde{N_u}(0,\eta)=0$ for $\eta\in X_{s}$. Moreover, $\hat{N_s}(\xi)=0$ for $\xi\in X_{u}$ because in this case $\eta=0$ and then $\hat{\xi}=\tilde{\xi}$.

To compute $\tn{Lip}(\hat{N_u})$ and $\tn{Lip}(\hat{N_s})$ let $x=\xi+\eta, y=\la+\mu$. Then
\begin{equation*}
\begin{array}{rl}
\|\hat{N_u}(x)-\hat{N_u}(y)\|=&\|\tilde{N_u}(\xi,\tilde{\theta}(\xi)+\eta)-\tilde{N_u}(\la,\tilde{\theta}(\mu)+\mu)\|\\
                         \leqslant & f(\gamma)(\|(\xi,\eta)-(\la,\mu)\|+\tn{Lip}(\tilde{\theta})\|\eta-\mu\|)\\
                         \leqslant & f(\gamma)(1+\tn{Lip}(\tilde{\theta}))\|(\xi,\eta)-(\la,\mu)\|=f(\gamma)(1+\tn{Lip}(\tilde{\theta}))\|x-y\|
\end{array}
\end{equation*}
and
\begin{equation*}
\begin{array}{rl}
\|\hat{N_s}(x)-\hat{N_s}(y)\|\leqslant & \|\tilde{N_s}(\xi,\tilde{\theta}(\xi)+\eta)-\tilde{N_s}(\la,\tilde{\theta}(\la)+\mu)\|+
\|\tilde{N_s}(\xi,\tilde{\theta}(\xi))-\tilde{N_s}(\la,\tilde{\theta}(\la))\|\\
                                + & \|\tilde{\theta}(\tilde{\xi})-\tilde{\theta}(\tilde{\la})\|+\|\tilde{\theta}(\hat{\xi})-\tilde{\theta}(\hat{\la})\|.
\end{array}
\end{equation*}

Analogously
\[
\|\tilde{N_s}(\xi,\tilde{\theta}(\xi)+\eta)-\tilde{N_s}(\la,\tilde{\theta}(\la)+\mu)\|+
\|\tilde{N_s}(\xi,\tilde{\theta}(\xi))-\tilde{N_s}(\la,\tilde{\theta}(\la))\|\leqslant 2f(\gamma)(1+\tn{Lip}(\tilde{\theta}))\|x-y\|,
\]
and hence
\[
\|\tilde{\theta}(\tilde{\xi})-\tilde{\theta}(\tilde{\la})\|\leqslant [b\tn{Lip}(\tilde{\theta})+f(\gamma)(1+\tn{Lip}(\tilde{\theta}))]\|\xi-\la\|
\leqslant[b\tn{Lip}(\tilde{\theta})+f(\gamma)(1+\tn{Lip}(\tilde{\theta}))]\|x-y\|.
\]

Thus
\begin{equation*}
\begin{array}{rl}
\|\tilde{\theta}(\hat{\xi})-\tilde{\theta}(\hat{\la})\|\leqslant & \tn{Lip}(\tilde{\theta})\|\hat{\xi}-\hat{\la}\|\\
\leqslant & \tn{Lip}(\tilde{\theta})[a^{-1}\|\xi-\la\|+f(\gamma)(\|(\xi,\eta)-(\la,\mu)\|+\tn{Lip}(\tilde{\theta})\|\xi-\la\|)]\\
\leqslant & \tn{Lip}(\tilde{\theta})[a^{-1}+f(\gamma)(1+\tn{Lip}(\tilde{\theta}))]\|x-y\|.
\end{array}
\end{equation*}
and, using that $\tn{Lip}(\tilde{\theta})\leqslant f_*(\gamma)$, we obtain
\begin{equation*}
\begin{array}{rl}
\tn{Lip}(\hat{N_s})\leqslant & 3f(\gamma)(1+\tn{Lip}(\tilde{\theta}))+b\tn{Lip}(\tilde{\theta})
+\tn{Lip}(\tilde{\theta})[a^{-1}+f(\gamma)(1+\tn{Lip}(\tilde{\theta}))]\\
\leqslant & f(\gamma)(1+f_*(\gamma))(3+f_*(\gamma))+f_*(\gamma)(b+a^{-1})=f_1(\gamma).
\end{array}
\end{equation*}
with $f_1(\gamma)\to 0$ as $\gamma\to 0^+$. So $S=L+\hat{N}$ with
$\tn{Lip}(\hat{N})\leqslant f_1(\gamma)$. In particular there exists $\gamma_1=\gamma_1(a,b)\in (0,\min\{\gamma_0,\gamma^*\}]$ such that for each $\gamma\in (0,\gamma_1]$, we have
\[
\|(I-L)^{-1}\|f_1(\gamma)<1\quad \mbox{ and } \quad b+2f_1(\gamma)<1<a-2f_1(\gamma).
\]

Hence $0$ is a $\mathfrak{L}$-hiperbolic equilibrium for $S$ .
Moreover, $W^u(S,0)=X_u$ and $W^s(S,0)=X_s$ which concludes this case.

{\sc $\ast$ Case 2:} $\delta<\infty$.

Let $T_1\in C(X)$ with $0$ as a weakly hyperbolic point and decomposition $T_1=L+N_1$ with parameters $\gamma,a,b,\delta$ and define $T=L+N$ with
\begin{equation}\label{eqhu3}
N(x)=
\left\{
\begin{split}
&N_1(x), \quad \quad \quad \hbox{ for } \|x\|\leqslant \delta\\
&N_1\left(\tfrac{\delta}{\|x\|}x\right),\quad \hbox{ for } \|x\| > \delta
\end{split}
\right.
\end{equation}
and $0$ is an weakly hyperbolic point with decomposition $T=L+N$ and parameters $\gamma,a,b,\infty$. Note that if $\gamma=\tn{Lip}(N_1)$ then $\tn{Lip}(N)\leqslant 2\gamma$. Define as in the first case $h\colon X_{u}\times X_{s}\to X$ by $h(\xi,\eta)=\xi+\sigma(\eta)+\eta$ and $S=g^{-1}\circ T\circ g$ with $k\colon X\to X_{u}\times X_{s}$ by $k(\xi+\eta)=(\xi,\tilde{\theta}(\xi)+\eta)$, $g=h\circ k$. Assume $0<\gamma\leqslant\gamma_1$ with $\gamma_1$ as before. Then $W^u(S,0)=X_u$ and $W^s(S,0)=X_s$. Defining $S_1=g^{-1}\circ T_1\circ g$. Now we show that there exists $\delta_1=\delta_1(a,b,\delta)>0$ such that $W^{s}_{\delta_1}(S_1,0)=W^{s}_{\delta_1}(S,0)$ and $W^{u}_{\delta_1}(S_1,0)=W^{u}_{\delta_1}(S,0)$.

Note that $\|g(\xi+\eta)\|=\|\xi+\sigma(\tilde{\theta}(\xi)+\eta)+\tilde{\theta}(\xi)+\eta\|\leqslant (1+\tn{Lip}(\tilde{\theta}))(1+\tn{Lip}(\sigma))\|\xi+\eta\|$ and defining
\[
\delta_1=\delta \left[\left(1+\dfrac{f(\gamma_1)}{a-b-3f(\gamma_1)}\right)\left(1+\dfrac{\gamma_1}{a-b-3\gamma_1}\right)\right]^{-1},
\]
with $f(\gamma_1)=\tfrac{2a\gamma_1}{a-b-3\gamma_1}$, we obtain $g(B_{\delta_1}^X(0))\subset B_{\delta}^X(0)$ and therefore $S|_{B_{\delta_1}^X(0)}={S_1}|_{B_{\delta_1}^X(0)}$, which concludes the result.

\subsection{Proof of Lemma \ref{convvarlipfinal_lem}} $ $ \label{Lemma52}

 We will present the proof for the unstable manifolds. The proof for the stable manifold is analogous and will be omitted. Taking $S_\eta(x)=T(x+x^\ast_\eta)-x_\eta^\ast$, for $x\in X$ and $\eta\in[0,1]$, we may assume that all the $\mathfrak{L}$-hyperbolic equilibria are $x^{\ast}_\eta=0$. Also, from the proof of Proposition \ref{teoconv}, we can assume that $T_\eta = L+N_\eta$ for each $\eta \in [0,1]$, where $L$ and $N_\eta$ satisfy the conditions of Definition \ref{hyperbolic_fixed_point} in $X$ for $\gamma,a,b>0$ independent of $\eta\in [0,1]$.

Applying a bi-Lipschitz change of variable in $X$, we can assume that $W^{u,0}_{loc}(0)=X_u$ and $W^{\rm s,0}_{loc}(0)=X_s$, and from the proof of Theorem \ref{teoUS}, there exists a family of maps $\theta_\eta\colon V_u\to X_s$ with ${\rm Lip}(\theta_\eta)<1$ such that $W^{\rm u,\eta}_{loc}(0)=\{\xi+\theta_\eta(\xi)\colon \xi \in V_u\}$, for sufficiently small $\eta$, with $\theta_0=0$, and
\begin{equation*}\label{X}
\left\{
\begin{split}
&h_{\eta}(\xi)=L_u\xi+N_{\eta,u}(\xi+\theta_\eta(\xi))\\
&\theta_{\eta}(h_{\eta}(\xi))=L_s\theta(\xi)+N_{\eta,s}(\xi+\theta_\eta(\xi))
\end{split}\right.\quad \hbox{ for }\xi \in V_u.
\end{equation*}

It remains to prove that $\norma{\theta_\eta}_{V_u}\to 0$ as $\eta\to 0^+$. From the proof of Theorem \ref{teoUS}, we know that $h_\eta\colon V_u\to V_u$ is invertible and
$$
\|h_\eta(\xi_1)-h_\eta(\xi_2)\|\geqslant (a-2\gamma)\|\xi_1-\xi_2\|, \hbox{ for all }\xi_1,\xi_2\in V_u \hbox{ and small }\eta.
$$

Since $N_{0,s}(\xi)=0$ for each $\xi\in V_u$ we have
\begin{align*}
\|\theta_\eta(h_\eta(\xi))\|&\leqslant b\|\theta_\eta(\xi)\|+\|N_{\eta,s}(\xi+\theta_\eta(\xi))-N_{0,s}(\xi+\theta_\eta(\xi))\|\\
& \qquad +\|N_{0,s}(\xi+\theta_\eta(\xi))-N_{0,s}(\xi)\|\\
&\leqslant (b+\gamma) \|\theta_\eta\|_{V_u,\infty}+\|N_{\eta}-N_{0}\|_{U,\infty},
\end{align*}
thus we have $\|\theta_\eta\|_{V_u,\infty} \leqslant \frac{1}{1-b-\gamma}\|N_{\eta}-N_{0}\|_{U,\infty}$ and therefore $\|\theta_\eta\|_{V_u,\infty} \to 0$ as $\eta\to 0^+$.

Now for $r_\eta = \|\theta_\eta\|_{V_u,\infty}$ and $K_\eta=\|N_{0,s}\|_{V_u\times \overline{B^{X_2}_{r_\eta}(0)}}+\|N_{\eta,s}-N_{0,s}\|_{U,Lip}$ we have
\begin{align*}
\|\theta_\eta(h_\eta(\xi_1))-\theta_\eta(h_\eta(\xi_2))\| & \leqslant
\|L_s(\theta_\eta(\xi_1)-\theta_\eta(\xi_2))+N_{\eta,s}(\xi+\theta_\eta(\xi))-N_{\eta,s}(\xi_2+\theta_\eta(\xi_2))\|\\
& \leqslant b\;\tn{Lip}(\theta_\eta)\|\xi_1-\xi_2\|+K_\eta\|(1+\tn{Lip}(\theta_\eta))\|\xi_1-\xi_2\|\\
& \leqslant [(b+K_\eta)\tn{Lip}(\theta)+K_\eta]\|\xi_1-\xi_2\|\\
& \leqslant \dfrac{(b+K_\eta)\tn{Lip}(\theta_\eta)+K_\eta}{a-2\gamma}\|h_\eta(\xi_1)-h_\eta(\xi_2)\|,
\end{align*}
and hence
$$
{\rm Lip}(\theta_\eta) \leqslant \frac{K_\eta}{a-b-2\gamma-K_\eta} \to 0 \hbox{ as } \eta\to 0^+,
$$
which concludes the result, since $K_\eta \to 0$ as $\eta\to 0^+$ from \eqref{hyp:N} and the convergence hypothesis on $T_\eta-T_0$.

\section{Autonomous perturbations of Lipschitz manifolds} \label{app:LipschitzMani}

In this section we deal with the question of perturbing Lipschitz manifolds. We begin with some preliminary results. For the rest of this section $I\colon X\to X$ will denote the identity map in $X$; that is, $Ix=x$ for each $x\in X$. We require that the reader take a look at Definition \ref{def:norms} to recall the norms that will be used.


\begin{lemma}\label{lem7.3}
If $g\colon X\to X$ and $\|g-I\|_{X,Lip}<1$ then $g\colon X\to X$ is bi-Lipschitz.
\end{lemma}
\begin{proof}
Clearly we have
$$
\|g(x)-g(y)\|\leqslant \|g(x)-g(y)-x+y\| +\|x-y\|\leqslant (\|g-I\|_{X,Lip} +1)\|x-y\|,
$$
and hence $g$ is a Lipschitz map. Now choose $0<\epsilon <1$ such that $\|g-I\|_{X,Lip}\leqslant \epsilon$. We have
\begin{equation} \label{eq:C1}
\|g(x)-g(y)\|\geqslant \|x-y\|-\|g(x)-g(y)-x+y\|\geqslant (1-\epsilon)\|x-y\|,
\end{equation}
and therefore $g$ is injective. If $y\in X$, define $h\colon X$ by $h(x)=y+x-g(x)$ for each $x\in X$. Thus
$$
\|h(x_1)-h(x_2)\| = \|x_1-x_2+g(x_2)-g(x_1)\|\leqslant \epsilon\|x_1-x_2\|,
$$
which proves that $h$ is a contraction and has a unique fixed point $x_0$ in $X$. This point satisfies $g(x_0)=y$ and hence $g$ is surjective. From \eqref{eq:C1} its inverse $g^{-1}$ is Lipschitz continuous and proves the result.
%
%
\end{proof}

\begin{proposition}\label{bilipschitzaberto}
Let $r>0$ and $g \colon \overline{B_r^X(0)}\to X$. If $\|g-I\|_{\overline{B_r^X(0)},Lip}<\tfrac{1}{2}$ then $g(B_r^X(0))$ is open and $g\colon B_r^X(0)\to g(B_r^X(0))$ is bi-Lipschitz. Moreover if  $\|g-I\|_{\overline{B_r^X(0)},\infty}\leqslant \alpha < 1$ we have $B_{r-\alpha}^X(0)\subset g(B_r^X(0))$.
\end{proposition}
\begin{proof}
Let $\tilde{g}\colon X\to X$ be defined by
\begin{equation*}
\tilde{g}(x)=
\left\{
\begin{matrix}
g(x),& \hbox{ if }\|x\|\leqslant r\\
g\left(\tfrac{rx}{\|x\|}\right)-\tfrac{rx}{\|x\|}+x, & \hbox{ if }\|x\|> r,
\end{matrix}
\right.
\end{equation*}
and choose $0<\epsilon<\tfrac 12$ such that $\|g-I\|_{\overline{B_r^X(0)},Lip}\leqslant \epsilon$.
If $x,y\in B_r^X(0)$ we have
$$
\|\tilde{g}(x)-\tilde{g}(y))-x+y\|=\|g(x)-g(y))-x+y\|\leqslant \epsilon \|x-y\|.
$$

On the other hand if $\|x\|,\|y\|\geqslant r$ we have
$$
\|\tilde{g}(x)-\tilde{g}(y)-x+y\|=\left\|g\left(\tfrac{rx}{\|x\|}\right) -g\left(\tfrac{ry}{\|y\|}\right) - \tfrac{rx}{\|x\|}+\tfrac{y}{\|y\|}\right\|
\leqslant \epsilon \left\|\tfrac{rx}{\|x\|}-\tfrac{ry}{\|y\|}\right\|,
$$
but
\begin{align*}
\left\|\tfrac{rx}{\|x\|}-\tfrac{ry}{\|y\|}\right\| \leqslant \frac{r}{\|x\|}\|x-y\|+r\|y\|\left|\frac{1}{\|x\|}-\frac{1}{\|y\|}\right| \leqslant 2\|x-y\|,
\end{align*}
and hence
$$
\|\tilde{g}(x)-\tilde{g}(y)-x+y\|\leqslant 2\epsilon\|x-y\|.
$$

Finally if $\|x\|\leqslant r$ and $\|y\|>r$ let $t\in[0,1]$ be chosen such that if $z=tx+(1-t)y$ then $\|z\|=r$. Hence
\begin{align*}
\|\tilde{g}(x)-\tilde{g}(y)-x+y\|& \leqslant  \|\tilde{g}(x)-\tilde{g}(z)-x+z\|+\|\tilde{g}(z)-\tilde{g}(y)-z+y\|\\
&\leqslant \epsilon \|x-z\|+2\epsilon\|z-y\|\leqslant 2\epsilon\|x-y\|,
\end{align*}
and therefore $\|\tilde{g}-I\|_{Lip}\leqslant 2\epsilon <1$. From Lemma \ref{lem7.3} we obtain $\tilde{g}$ bi-Lipschitz and $g(B_r^X(0))=\tilde{g}(B_r^X(0))$ is open.

For the last assertion note that $\|\tilde{g}(x)-x\|\leqslant \alpha$ for all $x\in X$. Now, since $\tilde{g}$ is bijective, given $y\in B_{r-\alpha}^X(0)$ there exists a unique $x\in X$ such that $\tilde{g}(x)=y$. But
$$
\|x\|\leqslant \|y\|+\|y-x\| < r-\alpha + \|\tilde{g}(x)-x\| \leqslant  r,
$$
thus $x\in B_r^X(0)$ and the result follows.
\end{proof}

\begin{corollary}\label{bilipschitzaberto2}
Let $U$ be an open subset of $X$ and $g\colon U\to X$. Se $\|g-I\|_{U,Lip}<\frac{1}{2}$ then $g(U)$ is open and $g\colon U\to g(U)$ is bi-Lipschitz.
\end{corollary}
\begin{proof}
For each $x\in U$ choose $r_x>0$ such that $\overline{B_{r_x}^X(x)}\subset U$. From the previous proposition $g(B_{r_x}^X(x))$ is open and $g\colon B_{r_x}^X(x)\to g(B_{r_x}^X(x))$ is bi-Lipschitz, and moreover the Lipschitz constants are independent of $x$. Hence $g$ is an open map, which shows that $g(U)$ is open. If there exist $x,y\in U$ with $x\neq y$ and $g(x)=g(y)$ we have
$$
\tfrac 12\|x-y\| > \left\|g(x)-g(y)-x+y\right\|= \left\|x-y \right\|
$$
which is a contradiction and proves that $g$ is injective. Moreover
\[
\|g(x)-g(y)\|=\|x-y+(-x+y + g(x)-g(y))\| >\tfrac 12 \|x-y\|,
\]
which proves that $g^{-1}$ is Lipschitz with ${\rm Lip}(g^{-1})<2$ and concludes the proof.
\end{proof}

%

\begin{proposition}\label{lem6.8}
Let $X_u,X_s$ be closed subspaces of $X$ with $X=X_u\oplus X_{s}$, $\pi_u$ the canonical projection of $X$ into $X_u$. Let $V_u\subset X_u$ be a neighborhood of $0$ in $X_u$, $\vp,\psi\colon V_u\to X$ maps with $\norma{\varphi-\psi}_{V_u}\leqslant \epsilon <\tfrac 12$. Assume that $\varphi(\xi) =\xi+\theta(\xi)$ for all $\xi\in V_u$ for some Lipschitz function $\theta\colon V_u\to X_s$ with $\theta(0)=0$. Then we have
\begin{itemize}
\item[\bf (a)] the set $W_u=\{\pi_u \psi(\xi)\colon \xi \in V_u\}$ is an open subset of $X_u$ containing $\psi(0)$;
\item [\bf (b)] if $B_{r}^{X_u}(0)\subset V_u$ and $r_0=r-\epsilon>0$ then $B_{r_0}^{X_u}(0)\subset V_u\cap W_u$;


\item[\bf (c)] there exists $\tilde{\theta}\colon W_u \to X_s$ such that $\psi(V_u) = \{\psi(0)+\eta+\tilde{\theta}(\eta)\colon \eta \in W_u\}$ with
\begin{equation} \label{eq:C41}
{\rm Lip}(\tilde{\theta}) \leqslant \frac{{\rm Lip}(\theta)+ \epsilon}{1-\epsilon}
\quad \hbox{ and } \quad
\|\tilde{\theta}-\theta\|_{V_u\cap W_u,\infty}\leqslant (1+{\rm Lip}(\theta))\epsilon.
\end{equation}
\end{itemize}

\end{proposition}
\begin{proof}
Let $\psi_j\colon V_u\to X_j$ for $j=u,s$ given by $\psi_u=\pi_u \circ \psi$ and $\psi_s=(I-\pi_u)\circ \psi$. We have $\norma{\psi_u-I}_{V_u}\leqslant \epsilon <\frac 12$ and Corollary \ref{bilipschitzaberto2} implies that $W_u=\psi_u(V_u)$ is open and $\psi_u\colon V_u\to W_u$ is bi-Lipschitz. Moreover from the proof of Corollary \ref{bilipschitzaberto2} we obtain that ${\rm Lip}(\psi_u^{-1})\leqslant \frac{1}{1-\epsilon}$. Item (b) is a direct consequence of Proposition \ref{bilipschitzaberto}. Also, with the same proposition, we conclude (a).

To prove (c), define $\tilde{\theta}\colon W_u \to X_s$ be given by $\tilde{\theta}(\xi)=\psi_s(\psi_u^{-1}(\xi))\hbox{ for }\xi \in W_u$.
Thus,
\[
\psi(V_u)=\{\psi_u(\xi)+\psi_s(\xi):\xi\in V_u\}=\{\eta+\tilde{\theta}(\eta):\eta\in W_u\}.
\]

Now
\begin{align*}
\psi_s(\xi_1)-\psi_s(\xi_2)-\theta(\xi_1)+\theta(\xi_2) = (I-\pi_u)\Big(\psi(\xi_1)-\psi(\xi_2)-\varphi(\xi_1)+\varphi(\xi_2)\Big),
\end{align*}
and hence ${\rm Lip}(\psi_s) \leqslant {\rm Lip}(\theta)+\epsilon$ which implies ${\rm Lip}(\tilde{\theta})\leqslant \frac{{\rm Lip}(\theta)+\epsilon}{1-\epsilon}$.

Now if $\xi\in V_u\cap W_u$ and $\tilde{\xi}=\psi_u^{-1}(\xi)\in V_u$ we have $\|\psi_u^{-1}(\xi)-\xi\|=\|\tilde{\xi}-\psi_u(\tilde{\xi})\| \leqslant \epsilon$, therefore
$$
\|\tilde{\theta}(\xi)-\theta(\xi)\| \leqslant \|\psi_s(\tilde{\xi})-\theta(\tilde{\xi})\|+\|\theta(\tilde{\xi})-\theta(\xi)\| \leqslant (1+{\rm Lip}(\theta))\epsilon,
$$
since $\|\psi_s-\theta\|_{V_u,\infty}\leqslant \norma{\psi-\varphi}\leqslant \epsilon$. Thus $\|\tilde{\theta}-\theta\|_{V_u\cap W_u,\infty}\leqslant (1+{\rm Lip}(\theta))\epsilon$.
\end{proof}

\subsection{Proof of Proposition \ref{prop6.9}} \label{app:B1}
The result is straightforward using Proposition \ref{lem6.8}, noting that item (b) guarantees the existence of a neighborhood $W_u$, independent of $\eta$, such that for sufficiently small $\eta$ the maps $\theta_\eta$ are defined in $W_u$.

\section{Differentiability of Nemytskii operators} \label{app:Diff}

We have based this sections in the results of \cite{Br} and \cite{Folland}, and here we prove basically that a differentiable Nemytskii operator from $L^p(\Omega)$ to $L^p(\Omega)$ of a real function must come from a affine function.  We begin with the Inverse Dominated Convergence Theorem and to this end consider $\Omega$ a bounded domain of $\mathbb{R}^n$ and $p\geqslant 1$.

\begin{theorem}[Inverse Dominated Convergence]\label{TICD}
Let $\{u_n\}$ be a sequence in $L^p(\Omega)$ and $u\in L^p(\Omega)$ such that $u_n\to u$ in $L^p(\Omega)$. Then there exist a subsequence $\{u_{n_k}\}$ of $\{u_n\}$ and a function $h\in L^p(\Omega)$ such that
\begin{itemize}
\item[\bf (i)] $u_{n_k}(x) \to u(x)$ a.e. in $\Omega$;
\item[\bf (ii)] $|u_{n_k}(x)| \leqslant h(x)$ for all $k$, a.e. in $\Omega$.
\end{itemize}
\end{theorem}
\begin{proof}
See \cite{Br}.
\end{proof}

\begin{lemma} Consider a continuous function $f\colon \mathbb{R}\to \mathbb{R}$ such that $|f(s)|\leqslant c(1+|s|^p)$ for all $s\in \mathbb{R}$, where $c\geqslant 0$ is a constant, and define the Nemytskii operator $f^e\colon L^p(\Omega)\to L^1(\Omega)$ associated with $f$ by $f^e(u)(x)=f(u(x))$ for each $x \in \Omega$ and $u\in L^p(\Omega)$. Then $f^e$ is well-defined and continuous.
\end{lemma}
\begin{proof} Let $u\in L^p(\Omega)$ and $\{u_n\}$ be a sequence in $L^p(\Omega)$ converging to $u$.
From Theorem \ref{TICD}, there exist a subsequence $\{u_{n_k}\}$ of $\{u_n\}$ and a function $h\in L^p(\Omega)$ such that $u_{n_k}(x) \to u(x)$ a.e. in $\Omega$ and $|u_{n_k}(x)| \leqslant h(x)$ for all $k$, a.e. in $\Omega$.

Hence $|f(u_{n_k}(x))|\leqslant c(1+|u_{n_k}(x)|^p) \leqslant C(1+|h(x)|^p)$ and from the Dominated Convergence Theorem we have
$$
\int_\Omega |f(u_{n_k}(x))-f(u(x))| dx \to 0 \quad \hbox{ as } k\to \infty,
$$
and since this limit does not depend on the sequence $\{u_n\}$ we obtain the continuity of $f^e$ in $u$.
\end{proof}

The following lemma has a straightforward proof and will help us ahead.
\begin{lemma}
If $f\colon \mathbb{R}\to \mathbb{R}$ is a differentiable and globally Lipschitz continuous function, its Nemytskii operator is well-defined from $L^p(\Omega)$ to $L^p(\Omega)$ and it is globally Lipschitz continuous.
\end{lemma}

We will also need the following result.

\begin{lemma}\label{Laux-Gateaux} If $f\colon \mathbb{R}\to \mathbb{R}$ is a differentiable and globally Lipschitz continuous function and $f^e$ is Fr\'{e}chet differentiable in $u_0\in L^p(\Omega)$ then
$[Df^e(u_0) h](x)=f'(u_0(x))h(x)$ for each $h\in L^p(\Omega)$, a.e. in $\Omega$.
\end{lemma}
\begin{proof} Since $f^e$ is Fr\'{e}chet differentiable in $u_0\in L^p(\Omega)$, for each $h\in L^p(\Omega)$ we have
$$
\lim_{t\to 0}\int_\Omega \left|\frac{f(u_0(x)+ th(x))-f(u_0(x))}{t}-[Df^e(u_0) h](x)\right|^p=0,
$$
and it follows that
$$
\lim_{t\to 0}\left\{\frac{f(u_0(x)+ th(x))-f(u_0(x))}{t}-[Df^e(u_0) h](x)\right\}=0, \hbox{ a.e. in } \Omega,
$$
which implies that $ [Df^e(u_0)h](x)=f'(u(x))h(x)$ a.e. in $\Omega$.
\end{proof}

We recall the Lebesgue Differentiation Theorem, that will be used to prove our main result.

\begin{theorem}[Lebesgue Differentiation Theorem]\label{LDT}
Let $g\in L^1_{\rm loc}(\Omega)$ and define
$$
A_rg(x)=\frac{1}{|B_r(x)|}\int_{B_r(x)} g(y) dy, \hbox{ for each } 0<r<{\rm dist}(x,\partial \Omega).
$$

Then  $A_rg(x) \to g(x)$ as $r\to 0^+$, a.e. in $\Omega$.
\end{theorem}
\begin{proof}
See \cite{Folland} for a proof of this result.
\end{proof}

We are now ready to state and proof the main result of this section.

\begin{theorem}\label{theo:Diff} If $f\colon \mathbb{R}\to \mathbb{R}$ is a differentiable and globally Lipschitz continuous function and its Nemystkii operator $f^e$ is Fr\'{e}chet differentiable at some point $u_0\in L^p(\Omega)$ then there exist $a,b\in \R$ such that $f(s)=a s + b$ for all $s\in \mathbb{R}$.
\end{theorem}
\begin{proof} Define $g_s (y)= |f(u_0(y)+s)-f(u_0(y))-f'(u_0(y))s|^p$ for each $s\in \mathbb{R}$ and $y\in \Omega$. From Theorem \ref{LDT} it follows that, fixed $s\in \mathbb{R}$, we have
$$
\lim_{r\to 0^+}A_r g_s(x)= g_s(x) \hbox{ for all } x\in \Omega\backslash E_s,
$$
where $E_s$ is a zero Lebesgue measure set. If $g_s(x)=0$ for all $s\in \R$ and a.e. in $\Omega$ we have $f(u_0(x)+s)=f(u_0(x))+f'(u_0(x))s$ for all $s\in \mathbb{R}$ and a.e. in $\Omega$ and the result is proved.

If for some $s_0\neq 0$ and $x_0\in \Omega\backslash E_{s_0}$we have  $g_{s_0}(x_0)\neq 0$,  let $r>0$ be such that $B_r(x_0)\subset \Omega$ and $u_r=s_0 \, \chi_{_{B_r(x_0)}}$. Since $f^e$ is Fr\'{e}chet differentiable at $u_0$, Lemma \ref{Laux-Gateaux} implies that $[Df^e(u)h](x)= f'(u(x))h(x)$ a.e. in $\Omega$ and we have
\begin{equation*}
\begin{aligned}
\frac{1}{\|u_r\|^p_{L^p(\Omega)}}\int_\Omega  & |f(u_0(y)+u_r(y))-f(u_0(y)) - f'(u_0(y))u_r(y)|^p dy\\
&=  \frac{1}{\|u_r\|^p_{L^p(\Omega)}} \int_{B_r(x_0)} |f(u_0(y)+s_0)-f(u_0(y)) - f'(u_0(y))s_0|^p dy\\
&=\frac{1}{|s_0|^p} \ A_{r}g_{s_0}(x_0) \to  \frac{g_{s_0}(x_0)}{|s_0|^p} \hbox{ as } r\to 0^+,
\end{aligned}
\end{equation*}
and since $g_{s_0}(x_0)\neq 0$ we obtain a contradiction with the Fr\'echet-differentiability of $f^e$ at $u_0$.
\end{proof}
\end{appendix}

\bibliographystyle{plain}
\bibliography{banco}

\end{document}